\documentclass[11pt]{amsart}
\usepackage{hyperref}
\usepackage{xypic}

\theoremstyle{plain} 
\newtheorem{theorem}{Theorem}[section]
\newtheorem{proposition}[theorem]{Proposition}
\newtheorem{corollary}[theorem]{Corollary}
\newtheorem{lemma}[theorem]{Lemma}
\newtheorem{conjecture}[theorem]{Conjecture}

\theoremstyle{definition} 

\newtheorem{example}[theorem]{Example}

\newtheorem{remark}[theorem]{Remark}
\newtheorem{definition}[theorem]{Definition}
\newtheorem*{mainthm}{Theorem}
\newtheorem{assumption}[theorem]{Assumption}

\DeclareMathOperator{\colim}{colim}

\DeclareMathOperator{\gr}{gr}
\DeclareMathOperator{\id}{id}

\DeclareMathOperator{\HC}{HC}
\DeclareMathOperator{\CC}{CC}
\DeclareMathOperator{\PHI}{\mathbf{\Phi}}
\DeclareMathOperator{\Aut}{Aut}
\DeclareMathOperator{\Perf}{Perf}
\DeclareMathOperator{\sfr}{sfr}
\DeclareMathOperator{\Sp}{sp}
\DeclareMathOperator{\fr}{fr}

\DeclareMathOperator{\nn}{\textbf{n}}

\DeclareMathOperator{\crit}{crit}
\DeclareMathOperator{\KK}{\mathrm{K_0}}
\DeclareMathOperator{\tr}{\mathop{tr}}
\DeclareMathOperator{\Gl}{GL}
\DeclareMathOperator{\coker}{coker}

\DeclareMathOperator{\re}{Re}
\newcommand{\Ho}{\mathrm{H}}

\newcommand{\spec}{\mathop{\rm Spec}\nolimits}

\newcommand{\nilp}{{\rm nilp}}

\renewcommand{\phi}{\varphi}

\newcommand{\Hom}{\mathop{\rm Hom}\nolimits}

\newcommand{\im}{\mathop{\rm Im}\nolimits}

\newcommand{\SL}{{\rm SL}}

\newcommand{\MHM}{\mathop{\rm MHM}\nolimits}
\newcommand{\MHS}{\mathop{\rm MHS}\nolimits}
\newcommand{\MHW}{\mathop{\rm MHW}\nolimits}

\newcommand{\Tr}{{\rm Tr}}
\newcommand{\Gr}{\mathop{\rm Gr}\nolimits}

\newcommand{\rat}{\mathop{\rm rat}\nolimits}

\newcommand{\pt}{\mathop{\rm pt}}

\renewcommand{\tilde}{\widetilde}
\newcommand{\Cp}{\mathbb{C}}

\renewcommand{\P}{\mathbb P}

\newcommand{\D}{\mathcal D}
\newcommand{\cN}{\mathcal N}
\newcommand{\N}{\mathbb N}

\newcommand{\C}{\mathbb C}

\newcommand{\Q}{\mathbb Q}

\begin{document}

\title[Purity and quantum cluster positivity]{Purity for graded potentials \\ and quantum cluster positivity}

\author[B. Davison ]{Ben Davison}
\address{B. Davison: GEOM, EPFL Lausanne}
\email{ndavison@epfl.ch}

\author[D. Maulik ]{Davesh Maulik}
\address{D. Maulik: Department of Mathematics, Columbia University}
\email{dmaulik@math.columbia.edu}

\author[J. Sch\"urmann ]{J\"org Sch\"urmann}
\address{J.  Sch\"urmann: Mathematische Institut, Universit\"at M\"unster}
\email{jschuerm@math.uni-muenster.de}

\author[B. Szendr\H oi ]{Bal\'azs Szendr\H oi}
\address{B. Szendr\H oi: Mathematical Institute, University of Oxford}
\email{szendroi@maths.ox.ac.uk}



\begin{abstract} Consider a smooth quasiprojective variety $X$ equipped with a $\C^*$-action, 
and a regular function $f\colon X\to\C$
which is $\C^*$-equivariant with respect to a positive weight action on the base.
We prove the purity of the mixed Hodge structure and the hard Lefschetz theorem on the cohomology of 
the vanishing cycle complex of $f$ on proper components of the critical locus of $f$, 
generalizing a result of Steenbrink for isolated quasi-homogeneous singularities. 
Building on work of Kontsevich--Soibelman, Nagao and Efimov, we use this result to prove the quantum positivity 
conjecture for cluster mutations for all quivers admitting a positively graded nondegenerate potential. 
We deduce quantum positivity for all quivers of rank at most~4; quivers with nondegenerate 
potential admitting a cut; and quivers with potential associated to triangulations of surfaces with marked points and nonempty boundary.
\end{abstract}
\maketitle
\tableofcontents
\thispagestyle{empty}

\section{Introduction}

Let~$X$ be a smooth quasiprojective variety of pure dimension~$n$, and let $f\colon X \to\C$
be a regular function.  Consider $\phi_f\Q_X[n]$, the perverse sheaf of 
vanishing cycles, supported on the critical locus~$Z$ of~$f$.
Its cohomology groups $H^{k}(Z, \phi_f\Q_X)$
naturally carry a mixed Hodge structure by classical work of Steenbrink and Navarro Aznar
\cite{St_limit, St_mhs, Nav}, or as an application of Saito's theory of mixed Hodge modules \cite{S1, S2}.  

In general, the weight filtration on the cohomology of vanishing cycles 
can be quite complicated. In the first part of this paper, we study this 
filtration when~$f$ admits a suitable $\C^*$-action and its critical 
locus is proper.

More precisely, assume that~$X$ carries a $\C^*$-action, so that~$f$ is equivariant with respect to the 
weight~$d$ action of~$\C^*$ on~$\C$, with $d>0$. Assume also that we have a proper subvariety~$Z\subset \{f=0\}$ 
that is is a union of connected components of the reduced critical locus $\{ df=0\}_{\rm red}\subset X$ of~$f$.
Under these hypotheses, we show the following purity and hard Lefschetz results (proven as Corollary~\ref{cor:pure}):
\begin{mainthm}
\begin{enumerate}
\item[(i)]
There is a direct sum decomposition
\[H^k(Z, \phi_f\Q_X|_Z) = H^k(Z, \phi_{f,1}\Q_X|_Z) \oplus H^k(Z, \phi_{f,\neq 1}\Q_X|_Z)\]
into pure Hodge structures of weights~$k$, $k-1$ respectively, where $\phi_{f,1}$ and $\phi_{f, \neq 1}$
are unipotent and non-unipotent vanishing cycles functors.
If~$d=1$, then the second piece vanishes. 
\item[(ii)]
If~$L$ is an ample line bundle on~$Z$, 
the natural cup product map defines a hard Lefschetz isomorphism
$$c_{1}(L)^k:  H^{n-k}(Z, \phi_{f,1}\Q_X|_Z) \to H^{n+k}(Z, \phi_{f,1}\Q_X|_Z)(k)$$
with an analogous statement for $\phi_{f,\neq 1}$.
\end{enumerate}
\end{mainthm}
While one can prove this statement using classical techniques, our proof will be a simple application of the theory of mixed Hodge modules, which also gives a more general statement.
It can be viewed as a higher-dimensional 
generalization of a classical result of Steenbrink~\cite{St} for quasi-homogeneous isolated singularities.
This result was conjectured in an early version of~\cite{Sz}, 
where potential applications to geometric engineering are discussed. We were
also inspired by a related conjecture and an example of Efimov~\cite{Efi11}. We can also 
view this result as proving the purity of part of the B-model noncommutative Hodge structure associated
to a graded Landau--Ginzburg model~\cite{KKP}, see especially [ibid., Remark 3.7]. 
For another application of the corresponding results for (unipotent) nearby cycles,  
see \cite[Remark 3.5]{BBM}, where the intersection space cohomology of a projective hypersurface 
with isolated singularities is studied.

In the second part of this paper, we combine this theorem with work of
Efimov~\cite{Efi11} on quantum cluster algebras.   We refer the reader to Section ~\ref{sec:Qcluspos} 
for basic definitions regarding cluster algebras, their quantum counterparts, and other terms that follow in this introduction.

Given a quiver~$Q$, assumed to have no loops or 2-cycles,
the \textit{quantum} cluster algebra~$A_{\Lambda, Q}$ is a noncommutative algebra built from $Q$ and the extra data of a 
compatible skew-symmetric matrix $\Lambda\in\mathrm{Mat}_{Q_0\times Q_0}(\mathbb{C})$, where $Q_0$ is
the set of vertices.  As in the classical context, distinguished generators of this algebra are conjectured to satisfy positivity properties.
Following work of Nagao~\cite{Nag13}, Efimov
uses the cohomological Hall algebra constructions of Kontsevich-Soibelman~\cite{KS10}
to reduce the quantum cluster positivity conjecture to a conjectural purity statement for
vanishing cohomology. 

In order to apply the first half of the paper, we work with quivers $Q$ which admit a potential $W$ in the usual sense that
is both nondegenerate (as in ~\cite{DWZ08}) and graded; that is, we impose the extra condition on $(Q,W)$ that
there is a grading of the edges of~$Q$ for which every term in~$W$ has fixed positive weight (as in \cite{AO10}). 
The following is our main result regarding quantum cluster algebras.

\begin{mainthm}
Let~$Q$ be a quiver admitting a graded nondegenerate potential, and~$\Lambda$ a compatible skew-symmetric matrix. Let $Y\in\mathcal{A}_{\Lambda,Q}$ be a quantum cluster monomial, and let $(Z_1,\ldots,Z_n)$ be a quantum cluster.  Then 
\[Y=\sum_{\mathbf{n}\in\mathbb{Z}^{Q_0}}a_{\mathbf{n}}(q)Z_{1}^{\mathbf{n}(1)}\cdots Z_{n}^{\mathbf{n}(n)},\] where all but finitely many of the~$a_{\mathbf{n}}(q)$ are zero, and the nonzero~$a_{\mathbf{n}}(q)$ are positive integer combinations of polynomials of 
the form 
\[P_{N,k}(q):=q^{\frac{N}{2}}(q^{\frac{-k}{2}}+q^{\frac{2-k}{2}}+\ldots+q^{\frac{k}{2}}),\] 
with~$N$, and the parity of~$k$, fixed by~$\mathbf{n}$.
\end{mainthm}
Given a quiver~$Q$, we may always find a compatible~$\Lambda$, possibly after expanding~$Q$ in a way that retains nondegeneracy of potentials (see Lemma~\ref{quantizationprop}, Remark~\ref{quantizationrem} and Proposition~\ref{adding_trees}). 
So we deduce the following application to the classical positivity conjecture of Fomin and Zelevinsky.
\begin{mainthm}
Let~$Q$ be a quiver admitting a graded nondegenerate potential, and let $\mathcal{A}_Q$ be the resulting 
commutative cluster algebra.  Let $Y\in\mathcal{A}_{Q}$ be a cluster monomial, and let $(Z_1,\ldots, Z_n)$ 
be a cluster.  Then 
\[Y=\sum_{\mathbf{n}\in\mathbf{Z}^{Q_0}}a_{\mathbf{n}}Z_1^{\mathbf{n}(1)}\cdots Z_n^{\mathbf{n}(n)},\] 
where all but finitely many of the $a_{\mathbf{n}}$ are zero, and all the $a_{\mathbf{n}}\in\mathbb{Z}_{\geq 0}$.
\end{mainthm}
\noindent This corollary is a special case of the main result in the recent paper of Lee--Schiffler~\cite{LS}, who
prove classical positivity for any quiver. Since the techniques are 
completely different, we nevertheless mention it explicitly.

While not all quivers allow graded nondegenerate potentials (Remark~\ref{rem:not_all}), 
we deduce the above strong form of quantum positivity for cluster algebras 
arising from the following objects: 
\begin{itemize}
\item quivers of rank at most 4;
\item
quivers mutation equivalent to an acyclic quiver, recovering the results of Kimura and Qin~\cite{FQ};
\item
quivers for which there is a nondegenerate potential admitting a cut, such as those coming from the Geiss--Leclerc--Schr\"oer construction~\cite{GLS08};
\item
dual quivers to ideal triangulations of surfaces with marked points and nonempty boundary, i.e. the cluster algebras of Fomin--Shapiro--Thurston~\cite{FST06, LF09}, using the potentials associated by Labardini-Fragoso \cite{LF09} to such quivers.
\end{itemize}
While we are unable to deduce positivity results for general quivers of rank 5 (Remark~\ref{5ormore}), our result implies quantum positivity for one well-studied example, the quiver coming from the Somos-5 sequence (Remark~\ref{5somos}).

\removelastskip\vskip.5\baselineskip\par\noindent{\bf Disclaimer.}
For the second half of the paper (Sections~\ref{sec:Qcluspos} and onwards), 
we use crucially the paper~\cite{Efi11}, which in turn extends the 
machinery of~\cite{KS08, KS10}. Some of these papers, and 
foundational work they rely on, remain in preprint form.
We comment on the necessary prerequisites further in Remark~\ref{remark:statusofKS}. 

\removelastskip\vskip.5\baselineskip\par\noindent{\bf Outline.}  We give a brief outline of the contents of this paper.  In Section $2$, we state some preliminary facts about mixed Hodge modules. In Section $3$, we use these to prove the main purity and Lefschetz results.  In Section $4$, we discuss background and applications: in Section $4.1$, we first give an application to Donaldson-Thomas invariants of quivers; in Sections $4.2$--$4.4$, we recall definitions and terminology regarding cluster algebras, mutations and categorification; in Section $4.5$, we review the work of Efimov relating positivity to purity. In order to apply these results in our setting, we require a certain result on families of nilpotent quiver representations, which we establish in Section $4.6$. Finally, in Section $4.7$, we prove positivity for quivers with graded potential and discuss the list of examples mentioned earlier. 

\removelastskip\vskip.5\baselineskip\par\noindent{\bf Conventions.}
 We are using the shifted convention for the (unipotent) vanishing and nearby cycle functors $\phi_f=\phi_{f,1}\oplus \phi_{f,\neq 1}$ and $\psi_f=\psi_{f,1}\oplus \psi_{f,\neq 1}$, mapping the perverse sheaf underlying a mixed Hodge module to a perverse sheaf. We will use the notation~$\phi_{f, \bullet}$ for either of~$\phi_{f,1}$ or~$\phi_{f,\neq 1}$ and similarly~$\psi_{f, \bullet}$ . Let $\epsilon(\bullet):=0$ for $\phi_{f, \bullet}=\phi_{f,1}$ and
$\epsilon(\bullet):=-1$ for $\phi_{f, \bullet}=\phi_{f,\neq 1}$ and either version of $\psi_{f, \bullet}$.
Tate twist as usual is denoted~$(-1)$. $\rat$ denotes the exact realization functor from mixed Hodge modules to perverse sheaves and also the corresponding derived functor. 

\removelastskip\vskip.5\baselineskip\par\noindent{\bf Acknowledgements.}
 We wish to thank Alex Dimca, Alexander Efimov, Bernhard Keller,
 Andr\'as Szenes, Daniel Labardini-Fragoso and Geordie Williamson for
 discussions, and an anonymous referee for spotting many
 inaccuracies. We also thank the American Institute of Mathematics for
 supporting, and the R\'enyi Institute of Mathematics, Budapest for
 hosting, a workshop in May 2012 where these ideas were first
 discussed.  During the preparation of this paper, BD was supported by
 Fondation Sciences Math\'{e}matiques de Paris and the DFG SFB/TR 45 
``Periods, moduli spaces and arithmetic of algebraic varieties''; 
DM was partially supported by NSF Grant DMS-1159416; JS was supported
by the  DFG SFB 878 ``Groups, geometry and actions''; BSz was
supported by EPSRC grant EP/I033343/1 and the Humboldt Foundation.   

\section{Preliminaries}

We start by recalling some results of M. Saito \cite{S1, S2}. 

\begin{theorem} Let $f\colon X\to \C$ be an algebraic morphism, and
$M\in \D^b\MHM(X)$ a pure (algebraic) mixed Hodge module complex of weight
$w$ (i.e.~~$Gr^W_jH^iM = 0$ for $j\neq i+w$). 
Fix $i\in \mathbb{Z}$. Assume one of the following. 
\begin{enumerate}
\item[a)] $f$ is proper, $Z= \{f=0\}$ and $H^i_\bullet=H^i(Z,\phi_{f,\bullet}(M))$.
\item[b)] $f$ is proper, $Z= \{f=0\}$ and $H^i_\bullet=H^i(Z,\psi_{f,\bullet}(M))$.
\item[c)] $f$ is not necessarily proper, but $Z\subset \{f=0\}$ is a proper union of connected components of
the support of $\phi_f(\rat(M))$, and $H^i_\bullet=H^i(Z,\phi_{f,\bullet}(M))$.
\end{enumerate}
In each case, we have a rational (graded polarizable) mixed Hodge structure 
\[H^i_\bullet \in \MHM(\pt)= \MHS^p\]
which carries an action of the nilpotent endomorphism 
\[N=\log(T_u)/2i\pi : (H_\bullet^i,F)\to (H_\bullet^i,F)(-1),\]
with $T_u$ being the unipotent part of the monodromy operator.

Then the weight filtration on $H^i_\bullet$ is the monodromy filtration of~$N$ 
shifted by $w'=w+i+\epsilon(\bullet)$,~i.e.
\[N^j\colon Gr^W_{w'+j}H_\bullet^i \to Gr^W_{w'-j}H_\bullet^i (-j)\]
is an isomorphism for all $j>0$. 

In particular, $H_\bullet^i$ is pure of weight $w'$ if and only if $N=0$, 
equivalently $T_u=id$, equivalently the monodromy $T$ acting on $H_\bullet^i$ is semisimple.
\label{thm_monodromy}
\end{theorem}

\begin{proof} In case a), 
using the properness of $f$, and noting that an algebraic mixed Hodge is automatically polarized, 
we have by \cite[Thm.2.14 on p.252]{S2} that 
\[H^i(Z,\phi_{f,\bullet}(M))= H^i(\{0\},\phi_{id,\bullet}(f_*M))
= H^0(\{0\},\phi_{id,\bullet}(R^if_*M))\]
On the other hand, by \cite[(4.5.2) and (4.5.4) on p.324]{S2}, 
pure Hodge module complexes are stable under direct images by the proper 
morphism $f\colon X\to \C$. Hence $f_*M$ is also pure of weight~$w$ and thus 
$R^if_*M$ is pure of weight~$w+i$. Then one can apply \cite[Prop.5.3.4, (5.3.4.2) on p.979]{S1}.

For case b), we argue analogously, using
\[H^i(Z,\psi_{f,\bullet}(M))= H^i(\{0\},\psi_{id,\bullet}(f_*M))
= H^0(\{0\},\psi_{id,\bullet}(R^if_*M))\:.\]

In case c), using the graph embedding and Nagata's compactification theorem, we can assume 
that~$f$ is the restriction of a proper complex algebraic morphism
$\tilde{f}\colon \tilde{X}\to \C$ to a Zariski open subset $j: X\hookrightarrow \tilde{X}$.
By~\cite[(4.5.4) on p.324]{S2}, we have a (non-canonical) decomposition 
\[M\simeq \oplus\: H^kM[-k] \in D^bMHM(X)\]
into pure Hodge modules. Applying the intermediate extension functor~$j_{!*}$ to each of 
the summands, we get pure Hodge modules
on the relative compactification $\tilde X$ by \cite[(4.5.2) on p.324]{S2}. 
Thus, we can assume that $M\simeq j^*\tilde{M}$ is the restriction to~$X$ of a pure mixed Hodge module 
complex~$\tilde{M}$ of weight~$w$ on~$\tilde{X}$.  
But then $H^i(Z,\phi_{f,\bullet}(M))$ is a direct summand of
$H^i(\{\tilde{f}=0\},\phi_{\tilde{f},\bullet}(\tilde{M}))$, since $Z\subset \{f=0\}$
is an open and closed subset of the support of $\phi_{\tilde{f}}(\rat(\tilde{M}))$.
This implies the claim by a).
\end{proof}

\begin{remark} In particular, recall that, when $X$ is smooth, the constant mixed Hodge module complex
$M=\Q^H_X\in D^bMHM(X)$ (with $\rat(\Q^H_X)= \Q_X$)  is pure of weight $w=0$. Moreover
$\Q^H_X[n]\in HM(X)$ in case $X$ is smooth and pure $n$-dimensional.
In this case, in case c), we can take $Z$ to be a proper  union of  connected components 
of the critical locus of $f$, since then the 
support of the sheaf $\phi_f(\Q_X)$ is just the critical locus of $f$.
 \end{remark}

\begin{theorem} (Hard Lefschetz)  Let $f\colon X\to \C$ be an algebraic morphism, and let
$M$ be a shifted pure Hodge module, i.e.~$M[n]\in HM(X)$. Assume one of the cases
{\rm a)--c)} of Theorem~\ref{thm_monodromy} above.
Assume also that for all $i$, the monodromy $T$ acting on $H_\bullet^i$ is semisimple,
so that $H_\bullet^i$ is pure of weight $w'=w+i+\epsilon(\bullet)$. 
Assume finally that $Z$ is projective with ample line bundle $L$.

Then, for $k > 0$, there exist isomorphisms of pure Hodge structures
\[ l^k\colon H_\bullet^{n-k}\to H_\bullet^{n+k}(k)\]
defined by $k$-fold cup product with $l=c_1(L)$.
\label{thm_lefsch}
\end{theorem}

\begin{proof} Let us assume that we are in case a) of Theorem~\ref{thm_monodromy}, so in particular
$Z=f^{-1}(0)$. 
Since $M[n] \in HM(X)$ by assumption, we can also assume (after a shift) that $M$ is a pure Hodge module.
Denote by $M'\in MHM(Z)$ either of the mixed Hodge modules $\phi_{f,\neq 1}(M)$ or $\phi_{f, 1}(M)$. 
By definition \cite[(5.1.6.2), p.956]{S1}, the weight filtration~$W$ of this mixed Hodge module
$M'$~is the monodromy filtration shifted by $w'$ of the nilpotent endomorphism
\[N=\log(T_u)/2i\pi : M'\to M'(-1).\]
This continues to hold in the larger abelian category $\MHW(Z)$ of~$W$-filtered Hodge modules, which 
contains $\MHM(Z)$ as a full subcategory.
This latter category is introduced in \cite[Lemma 5, p.854]{S1} in the case $Z$ smooth and 
more generally in \cite[p.237]{S2} for $Z$ singular (using local embeddings
and the smooth case from \cite{S1}). Here, $T_u$ again denotes the unipotent part of the monodromy; 
note that $N$ is an endomorphism from $M'$ to $M'(-1)$ by \cite[(5.1.3.4), p.953]{S1}.

Choose a polarization of the pure Hodge module $M$ of weight $w$, with underlying pairing 
\[S\colon \rat(M)\otimes \rat(M)\to  a^!_{X}\Q(-w)\]
as in \cite[Sec.5.2]{S1}. Here $a_X: X\to \pt$ is the constant map on $X$. 
This induces a pairing $S'=\phi_{f,\bullet} S$ on $M'$, 
with $\phi_{f,1} S$ resp. $\phi_{f,\neq 1} S =\psi_{f,\neq 1} S $ as in \cite[Sec.5.2]{S1} and
$w''=w+\epsilon(\bullet)$:
\[S'\colon \rat(M')\otimes \rat(M')\to a^!_Z\Q(-w'')\]
satisfying
\[S'(N\otimes id) + S'(id\otimes N)=0\]
by \cite[Lemma 5.2.5, p.965]{S1}. Then, by \cite[(5.2.10.2) and (5.2.10.3), p.968]{S1}, 
the $W$-filtered Hodge module $M'$ is strongly polarized by $N$ and $S'$ of weight $w''$ 
in the sense of \cite[Prop.1 on p.855]{S1}. 

Now take a closed embedding $Z\subset X'$ into a smooth projective variety $X'$ 
such that $L$ is the restriction of an ample line bundle on $X'$.
Then we can view (the pushforward of) $M'|Z$ as a $W$-filtered Hodge module on $X'$ with support in $Z$,
which is strongly polarized by $N$ and $S'$ of weight $w''$.
Therefore we can apply \cite[Prop.1(ii), p.855]{S1} to the constant map
$a_{X'}\colon X'\to \pt$ and $M'|Z$, to obtain the hard Lefschetz theorem for the cohomology 
of the nearby or vanishing cycles $M'|Z$ on $Z$.

The case b) is completely analogous, and so is c) using a compactification as in the proof of
Theorem~\ref{thm_monodromy} above.
\end{proof}

\section{Purity and hard Lefschetz results}

In what follows, let $X$ be a quasiprojective variety, equipped with 
a pure Hodge module complex $M\in \D^b\MHM(X)$ of weight $w$.  For 
a morphism $f\colon X\to \C$, let $R^if_*\rat(M)$ denote the
higher direct image sheaves of the underlying
constructible sheaf complex $\rat(M)$. Let $X_0=f^{-1}(0)$.

\begin{theorem} 
In the above setting, 
assume that $f\colon X\to \C$ is proper, and the restriction to
$\C^*\subset \C$ of $R^if_*\rat(M)$ is locally constant for all $i \in \mathbb{Z}$.
Then \[\phi_{id,\bullet}(f_*M)=(a_{X_0})_*(\phi_{f,\bullet}M) \in \D^b\MHM(\pt)\] is pure of
weight $w+\epsilon(\bullet)$, where $a_{X_0}: X_0\to\pt$ is the constant map. 

Assume further that $M$ is a shifted pure Hodge module, i.e.~$M[n]\in HM(X)$, and that
$Z\subset X_0$ is a projective union of connected components of
the support of $\phi_f(\rat(M))$, carrying an ample line bundle $L$.
Then the hard Lefschetz theorem holds with respect to $L$: we have isomorphisms
\[l^k:\:H^{n-k}(Z, \phi_{f,\bullet}M|_Z)\cong H^{n+k}(Z, \phi_{f,\bullet}M|_Z)(k) \]
defined by $k$-fold cup product with $l=c_1(L)$. 
\label{thm_general_purity}\end{theorem}

\begin{proof} Since $f$ is proper, one has as before by \cite[(4.5.2) on p.324]{S2} that also
$f_*M\in D^bMHM(\C)$ is pure of weight~$w$, i.e.~$R^if_*M$ 
is a pure Hodge module of weight $w+i$ for all $i \in\mathbb{Z}$. The assumption implies that the restriction of the perverse sheaf $rat(R^if_*M)$ to $\C^*$ is a locally constant sheaf (up to shift). Moreover
\[H^i(X_0,\phi_{f,\bullet}(M))= H^i(\{0\},\phi_{id,\bullet}(f_*M))
= H^0(\{0\},\phi_{id,\bullet}(R^if_*M))\:.\]
By Theorems~\ref{thm_monodromy} and~\ref{thm_lefsch}, we only have to show
that the action of the monodromy operator $T$ on the underlying complex vector spaces is semisimple.

Fix $i \in \mathbb{Z}$ and let $M'=R^if_*M\in HM(\C)$.
By the strict support decomposition of a pure Hodge module \cite[Sec.5.1]{S1},
$M'$ is the direct sum  of twisted intersection
complexes $IC_S(L)$, with $L$ a polarizable variation of pure rational Hodge
structures on $S=\{0\}$ or $S=\C^*$ (by \cite[Lem.5.1.10 on p.967]{S1}).
In case $S=\{0\}$, we have $\phi_{id}( IC_S(L))= IC_S(L)=L_{0}$ with $T$ acting as the identity, 
so that there is nothing to prove.

In case $S=\C^*$, we claim that $L$ is a locally constant polarizable variation of pure rational Hodge
structures.  Indeed, its pullback to the universal cover $\C$
of $\C^*$ the underlying sheaf is constant.  Therefore, one can apply
the rigidity theorem of Schmid \cite[Thm.7.22]{Sch} for such variation of Hodge structures
on a compactifiable complex manifold (like $\C$), which implies that the monodromy $T$ acting 
on the stalk $L_z$ ($z\in \C^*$) preserves the polarization as well as the Hodge filtration of 
the pure rational Hodge structure $L_z$. In particular, it acts as an isometry of a positive 
definite hermitian form on $L_z$.
But this implies that the monodromy action  on $L_z$ and therefore also the monodromy action $T$
on $\psi_{id}( IC_S(L))\simeq L_z$ ($z\in \C^*$) is semisimple.
But the canonical morphism ${\rm can}: \psi_{id}( IC_S(L))\to \phi_{id}( IC_S(L))$ is $T$-equivariant 
and surjective (\cite[Lem.5.1.4 on p.953-954]{S1}), so that also the action of $T$ 
on $\phi_{id}( IC_S(L))$ is semisimple.
\end{proof}

\begin{corollary} Let $f\colon X\to\C$ be a regular function on a smooth
quasiprojective variety $X$. Assume that 
\begin{enumerate}
\item[a)] $X$ carries a $\C^*$-action, so that $f$ is equivariant with respect to the weight~$d$ 
action of~$\C^*$ on~$\C$, for some $d>0$, and 
\item[b)] $Z\subset X_0$ is a compact (thus projective) union of connected components of 
the support of the reduced critical locus $\{ df=0\}_{\rm red}\subset X$ of~$f$. 
\end{enumerate}
Then there is a direct sum decomposition
\[H^k(Z, \phi_f\Q_X|_Z) = H^k(Z, \phi_{f,1}\Q_X|_Z) \oplus H^k(Z, \phi_{f,\neq 1}\Q_X|_Z)\]
into pure Hodge structures of weights $k$, $k-1$ respectively. 
If $d=1$, then the second piece vanishes. Also, if $L$ is an ample line bundle on~$Z$
and~$X$ is pure-dimensional, 
then the corresponding hard Lefschetz theorem holds with the appropriate shift. 
\label{cor:pure}\end{corollary}

\begin{proof} We first reduce to the case when $f$ is proper. Take a $\C^*$-equivariant 
compactification~$W$ of $X$, which exists since $X$ is quasiprojective.
Let $\Gamma$ denote the Zariski closure of the graph of $f$ inside 
$W\times\C$, whose projection to $\C$ is now proper. Finally, let $\tilde X$
be a $\C^*$-equivariant resolution of singularities of $\Gamma$, equipped with 
a proper map $\tilde f\colon\tilde X\to\C$. Then $X\subset\tilde X$ is 
an open set, and $\tilde f|_X = f$. The reduced critical locus of
$\tilde f$ contains $Z$ as an open and closed subset, since $Z$ is a proper 
union of components of the critical locus of $f$.
So we may assume that $f$ is proper as well as $\C^*$-equivariant. 

Now we can apply Theorem~\ref{thm_general_purity} above.  Finally
$$H^k(Z, \phi_{f,\neq 1}\Q_X|_Z)=H^k(Z, \psi_{f,\neq 1}\Q_X|_Z)$$ is a direct summand of
$$H^k(X_0, \psi_{f,\neq 1}\Q_X)=H^k(\{0\}, \psi_{id,\neq 1}Rf_*\Q_X)=0$$
in case $d=1$, since then the monodromy is acting trivially on $(Rf_*\Q_X)_z$
($z\in \C^*$) by the $\C^*$-action.
\end{proof} 

\begin{remark} Note that we do not claim that the Hodge module $\phi_f\Q_X|_Z$ itself is pure; this
is certainly not true in general. One example is discussed in~\cite[Sect.6]{Efi11}. In that example, 
the critical locus $Z$ is a union of three $\P^1$s joined at a point, and locally the function $f$ is 
given by $f(x,y,z)=xyz$. The module $\phi_f\Q_X|_Z$ has a nontrivial weight filtration, but its 
cohomology is pure. 
\end{remark}

\begin{remark} In Theorem~\ref{thm_general_purity} above, it is not necessary for $X$ to be smooth; for example, we can consider $M=IC^H_X$ in case~$X$ is pure dimensional. 
So in the proof of Corollary~\ref{cor:pure},
one can avoid equivariant
resolution of singularities, if one uses this pure intersection complex on an
equivariant (maybe singular) compactification. 
\end{remark} 

\begin{remark} Note that properness of the critical locus alone is not enough for
purity. Consider $f\colon\C^{n+1}\to \C$ with an isolated critical point at the origin. 
Then if $f$ is quasi-homogeneous, then by a classical result due to Steenbrink \cite[Thm.1]{St},
see also~\cite[Thm.7.1]{SS} and~\cite[Cor.5.5.5 on p.29]{Ku},
the monodromy acts semisimply on the (reduced) cohomology of the Milnor fibre, 
and hence the vanishing cohomology (at the origin) is pure, in accordance with our result Corollary~\ref{cor:pure}. 
However, if $f$ is not quasi-homogeneous, then this is not necessarily the case. 
An example is given by the $T_{p,q,r}$-singularity 
$$f(x,y,z):=x^p+y^q+z^r+axyz$$
with $a\neq 0$ and $p^{-1}+q^{-1}+r^{-1}<1$, for which the monodromy action on the 
(reduced) cohomology of the Milnor fibre at the origin is not semisimple; see \cite[Ex.9.1]{SS} and
\cite[Ex.7.3.5 on p.109]{Ku}.

On the other hand, there may be weaker conditions than the existence of a torus action which would allow
us to conclude purity. What we are really using is that certain local systems are locally constant over $\C^*$. 
This follows as long as the map $f$~is, or can be made, proper, with only one critical value.
\end{remark}

\section{Applications}

We refer to \cite{Kel10} for an excellent introduction to the following material.  Throughout this section, 
$Q$ will denote a finite quiver (directed graph), consisting of two sets $Q_0$ and $Q_1$, the vertices and 
the arrows of~$Q$ respectively, and two maps $s,t:Q_1\rightarrow Q_0$, giving for an arrow $a\in Q_1$ its source 
vertex $s(a)\in Q_0$ and its target vertex $t(a)\in Q_0$. A \textit{grading} of (the arrows of)~$Q$ is just a map 
of sets $Q_1\rightarrow \mathbb{Z}$. Note that any quiver admits the trivial grading which gives weight one to 
every arrow (i.e. $\Cp Q$ is graded by path length); this grading is often useful. When we call a quiver graded
below, we refer to an arbitrary grading, not necessarily the trivial one. Given a quiver $Q$, we denote its path 
algebra and its completion with respect to path length by $\Cp Q$ and $\widehat{\mathbb{C} Q}$. We may equivalently 
think of a quiver~$Q$ as being given by the data of a 
semisimple algebra $R=\oplus_{i\in Q_0}\Cp e_i$, generated by orthogonal idempotents, and an $R$-bimodule $T$ with a basis 
for $e_i Te_j$ provided by the arrows from $j$ to $i$. Then the (completed) path algebra is just the (completed) 
free unital tensor algebra of $T$ over $R$.

\subsection{Cohomological Donaldson--Thomas invariants of quivers}
\label{CohoDT}

One application, which originally lead to the above results in the form of conjectures, is discussed in detail
in~\cite{Sz}. 

Given a quiver~$Q$, a \textit{potential} on~$Q$ is a formal linear combination of cyclic words in the quiver~$Q$, or alternatively an element of the quotient of vector spaces $\widehat{\mathbb{C}Q} /[\widehat{\mathbb{C}Q},\widehat{\mathbb{C}Q}]$. We call a potential \textit{algebraic}, if it is a finite linear combination of cyclic words in the quiver, so can be considered as an element of $\mathbb{C}Q/[\mathbb{C} Q,\mathbb{C} Q]$ instead; note that the natural map $\Cp Q/[\Cp Q,\Cp Q]\rightarrow \widehat{\Cp Q}/[\widehat{\Cp{Q}},\widehat{\Cp{Q}}]$ is an injection. By QP, we will mean a pair $(Q,W)$ of a quiver with potential. We call a QP $(Q,W)$ algebraic if~$W$ is. 

If~$Q$ is graded, then the vector space $\widehat{\Cp Q}/[\widehat{\Cp Q},\widehat{\Cp Q}]$ is naturally graded, and a \textit{graded} potential~$W$ is a homogeneous element of $\widehat{\Cp Q}/[\widehat{\Cp Q},\widehat{\Cp Q}]$ which we will always assume to have strictly positive degree. We call $(Q,W)$ graded if $Q$ and $W$ are. A special case of a graded QP is that of a QP admitting a {\em cut}: the grading on~$Q$ takes values in $\{0,1\}$, and~$W$ is homogeneous of degree~$1$.   

Given a single cycle $u\in \mathbb{C}Q/[\mathbb{C}Q,\mathbb{C}Q]$, and $a\in Q_1$ an arrow of~$Q$, one defines 
\[
\frac{\partial u}{\partial a}=\sum_{\substack{b, c\text{ paths in }Q,\\\tilde{u}=bac}}cb
\]
where $\tilde{u}\in \Cp Q$ is a fixed lift of $u$.  One extends to a map $\frac{\partial}{\partial a}:\Cp Q/[\Cp Q,\Cp Q]\rightarrow \Cp Q$ by linearity.  
Given a (graded) algebraic QP, we define the (graded) \textit{Jacobi algebra}
\[
J(Q,W)=\Cp Q\Big/\Big\langle \frac{\partial W}{\partial a}\, \mathrm{\Big |}\, a\in Q_1\Big\rangle.
\]

Given a QP $(Q,W)$ with a marked vertex $0\in Q_0$, we define the {\em extended} or {\em framed} QP to be given by the pair $(\tilde Q, \tilde W)$, where $\tilde Q$ has one extra vertex $v\in\tilde Q_0$ with a single arrow to $0\in Q_0\subset\tilde Q_0$, and $\tilde W=W$. If $(Q,W)$ is graded, then $(\tilde Q, \tilde W)$ remains graded with trivial weight on the extra edge, and if $(Q,W)$ is algebraic, then $(\tilde Q, \tilde W)$ is too. We suppress the choice of the marked vertex from the notation, and assume, for the rest of this subsection, that $W$ is algebraic.

A dimension vector $\gamma=(\gamma_i)\in \N^{\Q_0}$ gives a dimension vector $\tilde\gamma=(1, \gamma_i)\in \N^{\tilde Q_0}$ 
on $\tilde Q$. 
Recall that a stability condition on a quiver in the sense of King~\cite{king} is given by the choice of a parameter
\[\theta \in \left\{ \theta \, | \, \theta(\tilde\gamma)=0\right\}\subset {\mathbb Z}^{\tilde Q_0} \subseteq \Hom\left(K_0(J(\tilde Q,W)\text{-mod}), {\mathbb Z}\right).\]
Let 
\[\cN_{\tilde Q, \tilde\gamma} =\bigoplus_{a\in \tilde Q_1}\Hom(\Cp^{\gamma_{t(a)}},\Cp^{\gamma_{s(a)}})\Big/\prod_{i\in Q_0}\Gl_{\Cp}(\gamma_{i}),
\]
where the action is via conjugation, be the stack of right representations of $\tilde Q$ with dimension vector~$\tilde\gamma$, 
framed at the vertex $v$,
and let \[\cN^\theta_{\tilde Q, \tilde\gamma}\subset \cN_{\tilde Q, \tilde\gamma}\] denote the open substack of $\theta$-stable 
representations. 
Then for generic $\theta$, the stack $\cN^\theta_{\tilde Q, \tilde\gamma}$ is in fact a smooth pure-dimensional 
quasiprojective variety, equipped with a regular function \[\tr(W)\colon \cN^\theta_{\tilde Q, \tilde\gamma} \to \C\] 
whose critical locus 
\[\cN^{\theta,W}_{\tilde Q, \tilde\gamma} =\{ d\tr(W)=0\}\subset \cN^\theta_{\tilde Q, \tilde\gamma} \]
is exactly the moduli space of $\theta$-stable representations of~$\tilde Q$ respecting the relations defined by
the noncommutative derivatives of~$W$. 
Then the cohomology groups 
\[\Ho^*\left(\cN^{\theta,W}_{\tilde Q, \tilde\gamma}, \phi_{\Tr W}\Q_{\cN^\theta_{\tilde Q, \tilde\gamma}}\right)\]
form (a version of) the {\em cohomological Donaldson--Thomas invariants}~\cite{DS, KS10} of the framed algebraic QP $(Q, W)$. 

With all these definitions, Corollary~\ref{cor:pure} above translates into

\begin{theorem} Assume that $(Q,W)$ is a graded algebraic QP. Assume also that the moduli space $\cN^{\theta,W}_{\tilde Q, \tilde\gamma}$ is 
projective. Then the mixed Hodge structure on $\Ho^k\left(\cN^{\theta,W}_{\tilde Q, \tilde\gamma}, \phi_{\Tr W}\Q_{\cN^\theta_{\tilde Q, \tilde\gamma}}\right)$ is a direct sum of two pure pieces of weights~$k$ and~$k-1$. If $(Q,W)$ admits a cut, then it is pure of weight~$k$. These cohomology groups also satisfy hard Lefschetz with respect to an ample line bundle. 
\label{thm:coDT}\end{theorem}

\begin{remark} The main motivation for introducing the above constructions is to find an algebraic analogue of studying moduli spaces of sheaves on certain local threefolds. The introduction of the extended quiver corresponds in these geometric situations to a framing of the sheaves considered. For $(Q,W)$ to admit a cut is quite common in geometric situations; for example, it holds for all QPs arising from consistent brane tilings~\cite{MR}. The assumption of properness in Theorem~\ref{thm:coDT} is strong, but also holds in certain interesting cases.
\end{remark}

\begin{example} Let~$Q$ be the conifold quiver with two vertices $Q_0=\{0,1\}$ and two arrows each from $0$ to $1$ and back denoted $a_{ij}$ and $b_{ij}$, with potential $W=a_{01}a_{10}b_{01}b_{10}-a_{01}b_{10}b_{01}a_{10}$. This QP is well known to admit a cut. Then, for dimension vectors $\tilde\gamma=(1,\gamma_0, \gamma_1)$ and suitable~$\theta$, the moduli spaces $\cN^{\theta, W}_{\tilde Q, \tilde\gamma}$ are various versions of rank-$1$ Donaldson--Thomas type moduli spaces of the resolved conifold geometry~\cite{NN}. For a more restrictive set of $\theta$, it follows from geometric considerations that these spaces are also proper. The purity of the cohomological DT invariants in this case has in~\cite{Sz} been connected via geometric engineering to a very different, $K$-theoretic purity of quiver moduli spaces attached to algebraic surfaces, whereas the Lefschetz action corresponds to a geometric $\SL(2)$-symmetry of the corresponding surface. 
\end{example}

\subsection{Mutations and the quantum cluster positivity conjecture}
\label{sec:Qcluspos}

A special case of the ideas of the previous subsection leads us to our main application in the theory of cluster algebras.  In this subsection, we introduce the relevant background material and state quantum cluster positivity and related conjectures.
While \cite{Kel10} remains a useful reference, we further follow the notations and conventions of \cite{Efi11}, where the link between positivity for quantum cluster algebras and purity of mixed Hodge structures arising in Donaldson--Thomas theory is first explored. 

Let us assume that~$Q$ has no loops (cycles of length 1) or oriented 2-cycles.  Given a labeling of $Q_0$ by the numbers $\{1\ldots, n\}$, we fix an integer $1 \leq m \leq n$ and define the $n\times m$ matrix $\tilde{B}$ by setting $b_{ij}=a_{ji}-a_{ij}$, where $a_{ij}$ is the number of arrows from the $i$th vertex of~$Q$ to the $j$th vertex.  The full subquiver generated by the vertices
$\{1,\ldots,m\}$ is denoted the principal part of $Q$.  We will be allowed to mutate (see below) at these vertices, but not at the others (the so-called ``frozen vertices''). 

Given a rank $n$ free $\mathbb{Z}$-module $L$ and a skew-symmetric form $\Lambda:L\times L\rightarrow \mathbb{Z}$, the quantum torus $T_{\Lambda}$ is a $\mathbb{Z}[q^{\pm 1/2}]$-algebra, freely generated as a $\mathbb{Z}[q^{\pm 1/2}]$-module by  elements $X^l$ for $l\in L$, with multiplication defined by $X^{e}\cdot X^{f}=q^{\Lambda(e,f)/2}X^{e+f}.$
Let $F_{\Lambda}$ denote the skew-field of fractions of $T_{\Lambda}$.  

A \textit{toric frame} is a map 
\[
M:\mathbb{Z}^n\rightarrow F_{\Lambda}
\]
of the form $M(c)=\phi(X^{\nu(c)})$ where $\phi\in \Aut_{\mathbb{Q}(q^{\pm 1/2})}(F_{\Lambda})$ and $\nu:\mathbb{Z}^n\rightarrow L$ is an isomorphism of lattices.

Let $\Lambda_M$ denote the skew-symmetric $n\times n$ integer matrix associated to $\Lambda$ via the isomorphism $\nu$.
We say that the pair $\tilde{B}$ and $\Lambda_{M}$ are \textit{compatible} if $$\tilde{B}^{T} \Lambda =\tilde{I},$$ where the first $m$ columns of $\tilde{I}$ are the $m\times m$ identity matrix $I_m$ and the rest of the columns are zero. 
In this case, the pair $(\tilde{B},M)$, or sometimes $(Q,M)$, is called a \textit{quantum seed}. 
The elements $M(1_{1}),\ldots,M(1_m)$ of $F_{\Lambda}$ are called the \textit{cluster variables}, while $M(1_{m+1}),\ldots,M(1_n)$ are the \textit{coefficients}. 

Without loss of generality, we can fix an identification $L = \mathbb{Z}^n$ at the start; given a skew-symmetric matrix compatible with $\tilde{B}$, we can take $\Lambda$ to be the associated skew form on $L$.  In this case, the toric frame $\nu = \id$ and $\phi = \id$ defines a quantum seed, which we take to be the \textit{initial seed}.

\smallbreak
Let $s\in\{1,\ldots,m\}$.  Then the mutation $\mu_s(Q)$ of~$Q$ at the vertex $s$ is defined as in \cite{FZ02}: first, for every path $ab$ of length 2 in~$Q$, where $b$ is an arrow from $j$ to $s$ and $a$ is an arrow from $s$ to $k$, we add a new arrow $[ab]$ from $j$ to $k$, and secondly we reverse all arrows $a$ from $s$ to $j$ to produce arrows $\overline{a}$ from $j$ to $s$, and likewise we replace all arrows $b$ from $j$ to $s$ with arrows $\overline{b}$ from $s$ to $j$.  The final part of the process of mutation is a \textit{deletion step} in that we cancel double arrows, i.e.~if, for any $i$ and $j$ in $Q_0$, there are $p$ arrows from vertex $i$ to vertex $j$, and $r$ arrows from $j$ to $i$, with $p\geq r$, we delete all the arrows from $j$ to $i$ and $r$ of the arrows from $i$ to $j$.  

We define mutation of quantum seeds as follows.  The mutation $\mu_s(\tilde{B})$ is defined as before: $\mu_s(\tilde{B})_{ij}$ is given by the number of arrows from the $j$th vertex of $\mu_s(Q)$ to the $i$th, minus the number from the $i$th vertex to the $j$th.  Finally,
\begin{align}
\mu_{s}(M)(1_i)=\begin{cases} M(1_i)& \mbox{for }i\neq s,\\
M(\sum_{b_{is}>0}b_{is}1_i-1_s)+M(-\sum_{b_{is}<0}b_{is}1_i-1_s) &\mbox{for }i=s.
\end{cases}
\end{align}
Specializing to $q^{1/2}=1$ we obtain the classical notion of cluster mutation.

In order for the above discussion to apply to a quiver $Q$, we require the existence of a compatible skew-symmetric matrix $\Lambda$ in order to construct quantum seeds.
We say the cluster algebra associated to $Q$ can be \textit{quantized} if there exists a quiver $Q' \supseteq Q$ 
for which we can find a compatible skew-symmetric matrix $\Lambda$, such that $Q'$ contains $Q$ as a full subquiver and the principal part of $Q'$ is the same as that of $Q$.  In other words, we can find a compatible matrix possibly after introducing extra cluster coefficients, or frozen vertices in the quiver language. 
\begin{lemma}[Quantization]\label{quantizationprop}
For every quiver~$Q$, the cluster algebra associated to~$Q$ can be quantized.
\end{lemma} 
\begin{proof}

Replace $Q$ with the quiver $Q'$ obtained as follows.  For each principal vertex $v$ of $Q$, we add a new frozen
vertex $v'$ and an arrow from $v$ to $v'$.  Let $B'$ denote the $(n+m) \times m$ matrix associated to $Q'$ as before; it consists
of $\tilde{B}$ concatenated with the $m\times m$ identity matrix.
Let $C$ denote the skew-symmetric $m \times m$ matrix obtained from the first $m$ rows of $\tilde{B}$.
Then 
$$\Lambda = \left(\begin{array}{ccc}0 & 0 & -\mathrm{id} \\ 0 & 0 & 0 \\ \id &0 & C \end{array}\right)$$
is compatible with $B'$.
\end{proof}

\begin{remark}\label{quantizationrem}
From the proof, we see that in order to quantize a quiver $Q$, it suffices to repeatedly perform the operation of gluing an acyclic quiver $T$, containing some vertex $i\in T_0$, to $Q$, containing a vertex $j$, by identifying $i$ with $j$.  
\end{remark}

We assume that an initial seed $(\tilde{B}, M)$ is given; if $\mathbf{s}$ is a sequence of vertices of~$Q$, we define $\mu_{\mathbf{s}}(\tilde{B})$ and $\mu_{\mathbf{s}}(M)$ recursively.  The set \[\{\{\mu_{\mathbf{s}}(M)(1_i)\mathrm{ | }\, i\in Q_0\},\mathbf{s}\text{ a sequence of vertices of }Q\}\] is called the set of \textit{quantum clusters} of $Q$.  The set
 \[
 \{\mu_{\mathbf{s}}(M)(\mathbf{n}), \mathbf{s}\text{ a sequence of vertices of }Q, \mathbf{n}\in\mathbb{Z}^{Q_0}_{\geq 0}\}
 \]
is called the set of quantum cluster monomials.  We denote by $\mathcal{A}_{\Lambda,Q}$ the $\mathbb{Z}[q^{\pm 1/2}]$-subalgebra of $F_{\Lambda}$ generated by the set 
 \[
 \{\mu_{\mathbf{s}}(M)(\mathbf{n}), \mathbf{s}\text{ a sequence of vertices of }Q, \mathbf{n}\in\mathbb{Z}^{Q_0}\text{ with }\mathbf{n}(i)\geq 0\text{ if }i\leq m\}.
 \]
The algebra $\mathcal{A}_{\Lambda,Q}$ is the \textit{quantum cluster algebra} associated to $\Lambda$ and~$Q$. Forgetting the data of~$\Lambda$, we form the classical cluster algebra $\mathcal{A}_Q$ by specializing $q^{1/2}=1$.
\begin{theorem}[Laurent phenomenon, \cite{BZ}, Corollary 5.2]
Given an arbitrary cluster monomial $Y\in \mathcal{A}_{\Lambda,Q}$, and an arbitrary quantum cluster $(Z_1,\ldots,Z_n)$, $Y$ is a Laurent polynomial in the $Z_i$, with coefficients $a_{\mathbf{n}}(q)$ in the ring $\mathbb{Z}[q^{\pm 1/2}]$.
\end{theorem}
The following conjecture, implicit in~\cite{BZ}, is a stronger version of the famous positivity conjecture of \cite{FZ02}, itself recently settled in~\cite{LS}. 
\begin{conjecture}[Quantum positivity]
In the above theorem, the polynomials $a_{\mathbf{n}}(q)$ belong to $\mathbb{Z}_{\geq 0}[q^{\pm 1/2}]$.
\end{conjecture}

The quantum positivity conjecture has been proved in the acyclic case in \cite[Cor.3.3.10]{FQ}, extending techniques originally used to prove the classical version of the acyclic case in \cite{Nak11}. The following stronger conjecture is introduced in \cite{Efi11} and is proved in the case in which the initial seed $(Q,M)$ has~$Q$ an acyclic quiver, under the additional assumption that $Y$ is a product of cluster variables from the initial seed, or the $Z_i$ come from the initial seed.
\begin{conjecture}[Lefschetz property]
\label{LefProperty}
In the above conjecture, each $a_{\mathbf{n}}(q)$ is a positive integral combination of polynomials of the form $P_{N,k}(q):=q^{\frac{N}{2}}(q^{\frac{-k}{2}}+q^{\frac{2-k}{2}}+\ldots+q^{\frac{k}{2}})$, where for each $\mathbf{n}$, $N$ and the parity of $k$ are fixed.
\end{conjecture}

\subsection{Mutation of quivers with potential}

In this section, we recall how to extend the notion of quiver mutation to quivers with potential, at least for QPs satisfying a certain nondegeneracy condition, and prove some basic lemmas.  The two statements that we will need for future sections are Proposition \ref{adding_trees}, which states that we can glue on acyclic quivers to an existing quiver $Q$ without effecting the nondegeneracy of potentials on $W$, which is a statement we need to make use of when we quantize, and Corollary \ref{formaltoalg}, which says that if we have a graded nondegenerate potential for a quiver $Q$ with respect to a sequence of vertices $\mathbf{s}$, we may find instead a graded \textit{algebraic} potential.

Start with a QP $(Q,W)$ consisting of a quiver $Q$ without loops or two-cycles, equipped with a possibly formal potential $W$. The \textit{premutation} $\mu_s'(Q,W)=(\mu_s'(Q),\mu_s'(W))$ of $(Q,W)$ is defined as follows: the underlying quiver $\mu_s'(Q)$ is defined in the same way as we defined mutation before, except that we skip the deletion step, i.e.~we do not remove double arrows (there may exist vertices $i$ and $j$ such that there are arrows from $i$ to $j$ and from $j$ to $i$). The premutation $\mu_s'(W)$ of~$W$ is defined by first defining $W_s$ to be the same linear combination of cyclic words as~$W$, except that for every path $ab$ of length two such that $s(a)=t(b)=s$, we replace every instance of $ab$ in every word of~$W$ with the arrow $[ab]$.  We then define
\begin{equation}
\label{W'def}
\mu_s'(W)=W_s+\sum_{\substack{{a,b\in Q_1}\\ s(a)=t(b)=s}}[ab]\overline{b}\overline{a}.
\end{equation}

Given two QPs, $(Q^{(1)},W^{(1)})$ and $(Q^{(2)},W^{(2)})$, with identical vertex set, their direct sum is defined as follows.
Heuristically, we draw both sets of arrows on the same set of vertices and take the corresponding sum of potentials.
More formally, we let $R=\oplus_{i\in Q_0}\Cp$ be the semisimple algebra defined by this vertex set, and $T^{(1)}$ and $T^{(2)}$ the $R$-bimodules corresponding to $Q^{(1)}$ and $Q^{(2)}$ respectively. We define 
\[(Q^{(1)},W^{(1)})\oplus (Q^{(2)},W^{(2)})=(Q^{(3)},W^{(3)}), \]
where $Q^{(3)}$ is given by the $R$-bimodule $T^{(1)}\oplus_{R\text{-bimod}} T^{(2)}$, and $W^{(3)}=W^{(1)}+W^{(2)}$. 

A QP $(Q,W)$ is called \textit{trivial}, if there is an isomorphism of $R$-algebras $\hat{J}(Q,W)\cong R$. It is called \textit{reduced}, if $W\in \prod_{p\geq 3}T^{\otimes_R p}$. Then we have the Splitting Theorem~\cite[Thm.4.6]{DWZ08},
according to which every QP $(Q,W)$ admits a splitting
\begin{equation}
\label{splitting}
(Q,W)\cong(Q_{\text{triv}},W_{\text{triv}})\oplus(Q_{\text{red}},W_{\text{red}})
\end{equation}
into a trivial and a reduced summand. This splitting is not unique, but the completed Jacobi algebra of the reduced part is well-defined up to isomorphism of $R$-algebras, with isomorphisms given by formal substitutions of variables taking arrows $a$ from $i$ to $j$ to $a_{\leq 1}+a_{>1}$, where $a_{\leq 1}=a$, and $a_{>1}$ is a linear combination of paths of length at least 2 from $i$ to $j$. 

Given this splitting construction, the \textit{mutation} $\mu_s(Q,W)$ of $(Q,W)$ is defined to be the reduced part 
of its premutation $\mu'_s(Q,W)$; while this does not define a unique choice of QP, 
the corresponding $\hat{J}(\mu_s(Q,W))$ is well defined up isomorphism given by substitutions as above. 
Note that even if $(Q,W)$ is an algebraic QP, it may not be possible to guarantee that $\mu_s(Q,W)$ can be represented
by an algebraic potential. 

We say that the potential~$W$ on the quiver~$Q$ is nondegenerate with respect to a vertex $s\in \{1,\ldots,m\}$, if $\mu_s(Q,W)$ can be represented by a reduced QP on a quiver which contains no 2-cycles. In this case, by \cite[Prop.7.1]{DWZ08} the underlying quiver of the mutated QP agrees with the mutated quiver $\mu_s(Q)$ defined before. We say that~$W$ is nondegenerate with respect to the sequence $\mathbf{s}=(s_1,\ldots,s_t)$ of vertices, if $(Q,W)$ can be mutated successively at the vertices $s_1,\ldots, s_t$ without producing a quiver with 2-cycles at any stage. Finally,~$W$ is just called \textit{nondegenerate}, if it is nondegenerate with respect to every sequence $\mathbf{s}$. Since we work over the uncountable field~$\Cp$, for every quiver~$Q$ there exists a nondegenerate algebraic potential (see \cite[Cor.7.4]{DWZ08}). 

\begin{lemma}\label{lem:rest:commutes} 
Let $Q$ be a full subquiver of an arbitrary quiver $Q'$. Given a potential $W$ on $Q'$, 
define the restriction operation $(Q',W)\mapsto(Q,W|_{Q})$, which sends the potential
\[
W=\sum_{l\text{ a cycle in }Q'}\alpha_l l
\]
to the potential
\[
W|_{Q}:=\sum_{l\text{ a cycle in }Q}\alpha_l l.
\]
Then for any vertex $i\in Q_0$ of $Q$, mutation at $i$ commutes with restriction.
\end{lemma}
\begin{proof}
Consider \[\mu_s'(W)=W_i+\sum_{\substack{a,b\in Q'\\s(a)=t(b)=i}}[ab]\overline{b}\overline{a}\]as in (\ref{W'def}), then
\[
W_i|_{Q}=(W|_{Q})_i
\]
and
\[
\sum_{\substack{a,b\in Q'\\s(a)=t(b)=i}}([ab]\overline{b}\overline{a})|_{Q}=\sum_{\substack{a,b\in Q\\s(a)=t(b)=i}}[ab]\overline{b}\overline{a}
\]
so that $\mu_s'(W)|_{Q}=\mu_s'(W|_{Q})$. By the Splitting Theorem~\cite[Thm.4.6]{DWZ08} there is a formal automorphism of $\Cp Q'$ sending each arrow $a\mapsto a+a_{>1}$, where $a_{>1}$ is a formal linear combination of paths of length at least two from $s(a)$ to $t(a)$, transforming $\mu_s'(W)$ to $\mu_s'(W)_{\text{triv}}+\mu_s'(W)_{\text{red}}$. We define a formal automorphism of $\Cp Q$ by sending $a\mapsto a_{\leq 1}+a_{>1}|_{Q}$.  It then follows that $\mu_s'(W)_{\text{red}}|_{Q}=(\mu_s'(W)|_{Q})_{\text{red}}$.
\end{proof}

For a QP $(Q,W)$  we define a two-term complex $\overline{\CC}_\bullet(Q,W)$ as follows. Let
\[\overline{\CC}_1(Q,W):=\widehat{\Cp {Q}}_{\text{cyc},\geq 1},\] 
be the space of formal linear combinations of cyclic paths of length at least one. Define also
\[
\overline{\CC}_2(Q,W):=\bigoplus_{a\in Q_1}e_{t(a)}(\widehat{\Cp Q}_{\geq 1}) e_{s(a)} \frac{\partial}{\partial a}.
\]
The map in the complex is defined by $W$ as 
\[\begin{array}{rcccl} d_{Q,W} & \colon & \overline{\CC}_2(Q,W) & \rightarrow & \overline{\CC}_1(Q,W)\\
\\
& & \displaystyle r\frac{\partial}{\partial a} & \mapsto & \displaystyle\sum_{W=uat}urt
\end{array}\]
for $r$ a path from $s(a)$ to $t(a)$.

We define $\overline{\HC}_1(Q,W)^*$ to be the cokernel of the map $d_{Q,W}$; it is the space of formal deformations of the potential $W$, modulo the infinitesimal action of the group of formal automorphisms of~$\widehat{\Cp Q}$.  The notation comes from the fact that this is the dual of the reduced cyclic homology of the Koszul dual category to $\hat{\Gamma}(Q,W)$, defined in Section \ref{Categorification_section}, which we remark is the space of first order deformations of this category as a strictly unital cyclic $A_{\infty}$ category with fixed CY pairing.

\begin{proposition}
\label{adding_trees}
Let $(Q,W)$ be a QP, with $W$ nondegenerate with respect to the sequence of vertices $\mathbf{s}$.  Let $Q'$ be obtained from $Q$ by first taking the union of $Q$ with a finite set of acyclic quivers $T_r$, for $r\in S$ a finite set, and then for each $r$ identifying one of the vertices of $T_r$ with one of the vertices of $Q$.  Then the QP $(Q',W)$ is nondegenerate with respect to $\mathbf{s}$, considered as a sequence of vertices of $Q'$.
\end{proposition}
\begin{proof} 
We prove the proposition by induction.  That is, we assume that $(Q',W)$ is nondegenerate with respect to $\mathbf{s}'$, for $\mathbf{s}'$ the sequence obtained from $\mathbf{s}$ by deleting the last vertex, and so we may define $\mu_{\mathbf{s}}(Q',W)$, and our job becomes to prove that the underlying quiver of $\mu_{\mathbf{s}}(Q',W)$ has no 2-cycles. 

Since gluing the $T_r$ to $Q$ introduces no new cycles, there is a natural isomorphism
\[
a\colon \overline{\HC}_1(Q',W)^*\stackrel{\sim}\longrightarrow\overline{\HC}_1(Q,W)^*.
\]
Since $W$ is non-degenerate with respect to $\mathbf{s}$, the underlying quiver of  $\mu_{\mathbf{s}}(Q,W)$ is $\mu_{\mathbf{s}}(Q)$.
Let $\mu_{\mathbf{s}}'(Q')$ denote the underlying quiver of $\mu_{\mathbf{s}}(Q',W)$; at this point we cannot assume that this is equal to $\mu_{\mathbf{s}}(Q')$.  As in \cite[Prop.7.3]{DWZ08}, there are maps $G:\widehat{\Cp Q}_{\text{cyc},\geq 1}\rightarrow \widehat{\Cp\mu_{\mathbf{s}}(Q)}_{\text{cyc},\geq 1}$ and $G':\widehat{\Cp{Q'}}_{\text{cyc},\geq 1}\rightarrow \widehat{\Cp\mu_{\mathbf{s}}'(Q')}_{\text{cyc},\geq 1}$. The map $G$ takes a potential for $Q$, nondegenerate with respect to $\mathbf{s}'$, to the corresponding potential on the underlying quiver  $\mu_{\mathbf{s}}(Q)$ of $\mu_{\mathbf{s}}(Q,W)$, sending $W$ to the potential for $\mu_{\mathbf{s}}(Q,W)$; the map $G'$ acts similarly for $Q'$. Taking derivatives at $W$, we obtain morphisms $\overline{\CC}_1(Q',W)\rightarrow\overline{\CC}_1(\mu_{\mathbf{s}}(Q',W))$ and $\overline{\CC}_1(Q,W)\rightarrow\overline{\CC}_1(\mu_{\mathbf{s}}(Q,W))$.
Quotienting by the infinitesimal actions of groups of formal automorphisms, we then obtain natural homomorphisms \[b':\overline\HC_1(Q',W)^*\stackrel{\sim}\longrightarrow \overline\HC_1(\mu_{\mathbf{s}}(Q', W))^*\] and \[b:\overline\HC_1(Q,W)^*\stackrel{\sim}\longrightarrow \overline\HC_1(\mu_{\mathbf{s}}(Q,W))^*.\]  
We next claim that these maps are isomorphisms. Indeed, their inverses are given by pushing a formal deformation of the potential $\mu_{\mathbf{s}}(W)$ along the reverse sequence of mutations.  Here we have used the fact that mutation defines an involution on (equivalence classes of) QPs \cite[Thm.5.7]{DWZ08}. Note that the only possible two-cycles in~$\mu_{\mathbf{s}}'(Q')$ appear at the last step by our induction hypothesis, so going backwards, the first mutation is not at a vertex where there may be possible two cycles. Hence the backwards maps are defined just as the forwards maps above, and they become natural inverses.

Furthermore, using the fact that we may take the potential for $\mu_{\mathbf{s}}(Q,W)$ to be the restriction to $Q$ of the potential for $\mu_{\mathbf{s}}(Q',W)$ by Lemma~\ref{lem:rest:commutes}, there is a commutative diagram
\[
\xymatrix{
\overline{\CC}_2(\mu_{\mathbf{s}}(Q',W))\ar[rr]^{d_{\mu_{\mathbf{s}}(Q',W)}}\ar[d]&&\overline{\CC}_1(\mu_{\mathbf{s}}(Q',W))\ar[d]\\
\overline{\CC}_2(\mu_{\mathbf{s}}(Q,W))\ar[rr]^{d_{\mu_{\mathbf{s}}(Q,W)}}&&\overline{\CC}_1(\mu_{\mathbf{s}}(Q,W))\\
}
\]
where the vertical arrows are given by restriction, inducing a map 
\[\textbf{res}:\overline{\HC}_1(\mu_{\mathbf{s}}(Q',W))^*\rightarrow\overline{\HC}_1(\mu_{\mathbf{s}}(Q,W))^*.\]  
The diagram 
\[
\xymatrix{
\overline{\HC}_1(Q',W)^*\ar[d]^-{a}\ar[r]^-{b'}&\overline{\HC}_1(\mu_{\mathbf{s}}(Q',W))^*\ar[d]^{\textbf{res}}\\
\overline{\HC}_1(Q,W)^*\ar[r]^-{b}&\overline{\HC}_1(\mu_{\mathbf{s}}(Q,W))^*
}
\]
commutes, from which we deduce that the restriction map $\textbf{res}$ is an isomorphism, as all the other maps are.

Now let us assume, for a contradiction, that $\mu'_{\mathbf{s}}(Q')$, the quiver underlying $\mu_{\mathbf{s}}(Q', W)$,
contains a 2-cycle. By Lemma~\ref{lem:rest:commutes}, 
the two vertices of this two-cycle cannot both be contained in $Q_0$, since $W$ is assumed nondegenerate with respect 
to $\mathbf{s}$. So the two-cycle must pass through some vertex $i$ for $i\in Q'_0\setminus Q_0$.  On the other hand, from the definition of $\overline{\HC}_1(Q',W)^*$ it is immediate that this 2-cycle represents a nonzero class, as $\text{Im}(d_{\mu_{\mathbf{s}}(Q',W)})\subset \widehat{\Cp\mu_{\mathbf{s}}'(Q)}_{\text{cyc},\geq 3}$. But this class is obviously killed by $\textbf{res}$, a contradiction. 

We deduce that $\mu_{\mathbf{s}}(Q')$ contains no 2-cycles, completing the induction step.
\end{proof}

We now describe how to pass from formal nondegenerate potentials to algebraic ones.  Consider the trivial grading on $\Cp Q$ given by path length. Denote by $W_n$ the graded piece of~$W$ corresponding to paths of length $n$, and let $W_{\leq n}:=\sum_{i\leq n} W_i$.

\begin{lemma}\cite[Prop.4.15]{DWZ08}
\label{twolemma}
Let~$Q$ be a quiver without loops, and let~$W$ be a formal potential on~$Q$.  Then the question of whether $(Q_{\mathrm{red}},W_{\mathrm{red}})$ has no 2-cycles is settled by $W_{2}$.
\end{lemma}
\begin{lemma}
\label{trunclemma}
Let $Q$ be a quiver, and let $\mathbf{s}$ be a sequence of vertices of $Q$.  For each $n\in \mathbb{N}$ there exists $m\in \mathbb{N}$ such that for every formal potential $W$ on $Q$, which is nondegenerate with respect to $\mathbf{s}$, the value of $\mu_s(W)_{\leq n}$, up to formal automorphism, is determined by $W_{\leq m}$.
\end{lemma}
\begin{proof}
Indeed from the description of $W'$ from (\ref{W'def}), and the fact that for every formal automorphism $\psi:\widehat{\Cp Q}\rightarrow \widehat{\Cp Q}$, $\psi(W)_{\leq t}$ is determined by $W_{\leq t}$, we deduce that it is enough to assume $m\geq (3/2)^t\cdot n$, where $t$ is the length of $\mathbf{s}$.
\end{proof}
Putting these lemmas together, we deduce that the condition of being nondegenerate with respect to a fixed sequence $\mathbf{s}$ is an open condition in the $\Cp Q_{\geq 1}$-adic topology, where $\Cp Q_{\geq 1}$ is the ideal generated by paths of length at least one.  

Assume now that $(Q,W)$ is a graded QP (with respect to an arbitrary grading of the arrows). 
Then for any vertex $s$, the premutated QP $\mu_s'(Q,W)$ is also graded, after introducing the following grading conventions (\cite[Def.6.4]{AO10}) on mutated quivers: 
\begin{itemize}
\item $|[ab]|=|a|+|b|$,
\item $|\overline{a}|=|W|-|a|$ if $s(a)=s$, 
\item $|\overline{a}|=-|a|$ if $t(a)=s$.
\end{itemize}
As proved in \cite[Thm.6.6]{AO10}, the process of passing from a QP to a reduced QP preserves any $\mathbb{Z}$-grading, and so we may define the mutation of a graded QP in the same way. 

We finally obtain the corollary that enables us to use \textit{algebraic} potentials in all our applications. 

\begin{corollary}
\label{formaltoalg}
Let $(Q,W)$ be a QP, and assume that the formal potential~$W$ is nondegenerate with respect to the sequence of vertices $\mathbf{s}$.  Then there exists an $n\in \mathbb{N}$ such that $W_{\leq n}$ is nondegenerate with respect to $\mathbf{s}$ too.  In particular, if~$Q$ has a nondegenerate graded potential $W$, $Q^{\diamondsuit}$ is mutation equivalent to~$Q$, and $\mathbf{s}$ is a sequence of vertices of $Q^{\diamondsuit}$, then there exists an algebraic graded potential~$W^{\diamondsuit}$ for $Q^{\diamondsuit}$ that is nondegenerate with respect to $\mathbf{s}$.
\end{corollary}
\begin{proof}
For the first assertion, we may pick $n=\lceil 3^t/2^{t-1}\rceil$, where $t$ is the length of $\mathbf{s}$, by Lemmas~\ref{twolemma} and~\ref{trunclemma}.  For the last assertion, let $\mathbf{r}$ be a sequence of vertices of $Q$ such that $\mu_{\mathbf{r}}(Q)=Q^{\diamondsuit}$.  Then $\mu_{\mathbf{r}}(Q,W)$ is a graded QP, with nondegenerate graded potential $W^{\diamondsuit}$, and underlying graded quiver $Q^{\diamondsuit}$.  We then use the first assertion to truncate the potential $W^{\diamondsuit}$ to an algebraic graded potential, nondegenerate with respect to the sequence $\mathbf{s}$.
\end{proof}

\subsection{Categorification}
\label{Categorification_section}
We now to turn to categorification of quantum cluster mutation.
Given an algebraic QP $(Q,W)$, recall the Jacobi algebra $J(Q,W)$ defined in Subsection~\ref{CohoDT}. In \cite[Sec.5]{Gin07}, Ginzburg defines $\Gamma(Q,W)$, a nonpositively graded dg algebra as follows.  The underlying graded algebra is given by $\Cp\tilde{Q}$, where $\tilde{Q}$ is a graded quiver constructed in the following way. We first identify the vertices of $\tilde{Q}$ with those of $Q$.  The degree zero arrows of $\tilde{Q}$ are identified with those of $Q$.  In degree $-1$, for each arrow $a$ in~$Q$ from $i$ to $j$, we add an arrow $a^*$ from $j$ to $i$ in $\tilde{Q}$.  Finally, for each vertex $i$ of~$Q$ we add a loop $\omega_i$ to $\tilde{Q}$, in degree $-2$, based at $i$.  The differential $d$ of $\Gamma(Q,W)$ is a derivation, defined on generators by
\begin{align*}
d(a)=&\ 0,\\
d(a^*)=&\ \frac{\partial W}{\partial a},\\
d(\omega_i)=&\sum_{a\in Q_1}e_i[a,a^*]e_i.
\end{align*}
By definition, we have that $\Ho^0(\Gamma(Q,W))\cong J(Q,W)$.  If $(Q,W)$ is instead a graded algebraic QP, there is a natural bigrading on $\Gamma(Q,W)$ such that the differential is of degree $(0,1)$, and the zeroth homology, with respect to the second grading, is the graded Jacobi algebra $J(Q,W)$.  

Completing the underlying graded algebra of $\Gamma(Q,W)$ with respect to the number of arrows of~$Q$ occurring in a path, and imposing the same differential $d$, we obtain the completed Ginzburg algebra $\hat{\Gamma}(Q,W)$, with $\Ho^0(\hat{\Gamma}(Q,W))=\hat{J}(Q,W)$.  If we consider the completed path algebra $\widehat{\C{Q}}$, $\hat{J}(Q,W)$ is the quotient by the closure of the ideal generated by the noncommutative derivatives of $W$. Note that in contrast with the case that~$W$ is an algebraic potential, $\hat{\Gamma}(Q,W)$ and $\hat{J}(Q,W)$ are well-defined for a non-algebraic $(Q,W)$.
\smallbreak
For an algebra $A$ we denote by $A\mathrm{-Mod}$ the category of right modules for $A$, and by $A\mathrm{-mod}$ the category of finite dimensional right $A$ modules.  Similarly for $\Gamma$ a dg algebra we denote by $\Gamma\mathrm{-Mod}$ the category of right dg modules for $\Gamma$, and by $\Gamma\mathrm{-mod}$ the category of right dg modules for $\Gamma$ with finite dimensional total homology.  

\smallbreak
We now recall the details of the connection between the above categorification and quantum cluster mutation.  Recall that $\Perf(\hat{\Gamma}(Q,W))$ is the smallest strictly full subcategory of the derived category of $\hat{\Gamma}(Q,W)\text{-Mod}$ containing the modules $e_i\hat{\Gamma}(Q,W)$, for $i\in Q_0$, stable under shifts, extensions, and direct summands. Given a QP $(Q,W)$ which is nondegenerate with respect to the sequence $\mathbf{s}=(s_1,\ldots,s_t)$, and an arbitrary sequence of signs $\epsilon$ of length $t$, it is proved in \cite[Thm.3.2, Rem.3.3]{KY11} that there is a quasi-equivalence of dg categories
\[
\Phi_{\mathbf{s},\epsilon}:\Perf(\hat{\Gamma}(Q,W)\mathrm{-Mod})\xrightarrow{\sim} \Perf(\hat{\Gamma}(\mu_{\mathbf{s}}(Q,W))\mathrm{-Mod})
\]
between the dg categories of perfect dg modules for the respective (completed) Ginzburg dg algebras, which restricts to a quasi-equivalence
\[
\Phi_{\mathbf{s},\epsilon}:\hat{\Gamma}(Q,W)\textrm{--mod}\xrightarrow{\sim}\hat{\Gamma}(\mu_{\mathbf{s}}(Q,W))\textrm{--mod}
\]
between the dg categories of dg modules with finite dimensional total cohomology.

  Now, following \cite{Nag13} we recursively define a torsion structure $T_{\mathbf{s}}$ on the abelian category $\hat{J}(Q,W)\text{-mod}$, and recursively define a choice $\epsilon_{\mathbf{s}}$ for $\epsilon$.  Denote by $\mathbf{s}_{\leq q}$ the sequence $(s_1,\ldots,s_q)$.  Then $T_{\mathbf{s}_{\leq 0}}$ is the full subcategory containing the zero module.  Denote by $S_{\mathbf{s}_{\leq q},i}$ the 1-dimensional simple $\hat{J}(\mu_{\mathbf{s}_{\leq q}}(Q,W))$-module concentrated at the $i$th vertex, and denote by $S_{\mathbf{s},q}$ the module $\Phi^{-1}_{\mathbf{s}_{\leq q-1},\epsilon_{\epsilon_{q-1}}}(S_{\mathbf{s}_{\leq q-1},\mathbf{s}_q})$.  Recall that for two subcategories $\mathcal{C}_1$ and $\mathcal{C}_2$ of an abelian category $\mathcal{C}$, $\mathcal{C}_1\star \mathcal{C}_2$ is defined to be the full subcategory of $\mathcal{C}$ containing objects $M$ that fit inside 
short exact sequences $M_1\rightarrow M\rightarrow M_2$ for $M_i\in\mathcal{C}_i$, $i\in \{1,2\}$.  For the recursive step, we define 
\begin{align}
\label{recdef1}
T_{\mathbf{s}_{\leq q}}=&\begin{cases} T_{\mathbf{s}_{\leq q-1}}\star S_{\mathbf{s},q}^{\oplus} &\text{if } S_{\mathbf{s},q}\notin T_{\mathbf{s}_{\leq q-1}},\\ 
T_{\mathbf{s}_{\leq q-1}}\cap  (^{\perp}S_{\mathbf{s},q}) &\text{otherwise.} \end{cases}
\end{align}
where $S_{\mathbf{s},q}^{\oplus}$ is the full subcategory of $\hat{J}(Q,W)\text{-mod}$ containing the objects $S_{\mathbf{s},q}^{\oplus n}$ for $n\in\mathbb{N}$.  In addition, we let $\epsilon_{\leq q}$ be obtained from $\epsilon_{\leq q-1}$ by adding a $+$ to the end in the first case of (\ref{recdef1}), and a $-$ otherwise.  Define $F_{\mathbf{s}}=T_{\mathbf{s}}^{\perp}\subset\hat{J}(Q,W)\mathrm{-mod}$, and define the abelian category $\hat{J}(Q,W)\mathrm{-mod}^{(T_{\mathbf{s}}[-1],F_{\mathbf{s}})}$ to be the full abelian subcategory of the derived category of $\hat{J}(Q,W)$-modules consisting of objects $M$ such that $\Ho^{1}(M)\in T_{\mathbf{s}}$, $\Ho^0(M)\in F_{\mathbf{s}}$ and $\Ho^i(M)=0$ for $i\neq 0,1$.
\begin{proposition}\cite[Thm.3.4]{Nag13}
There is an equality
\[
\Phi_{\mathbf{s},\epsilon_{\leq t}}(\hat{J}(Q,W)\mathrm{-mod}^{(T_{\mathbf{s}}[-1],F_{\mathbf{s}})})=\hat{J}(\mu_{\mathbf{s}}(Q,W))\mathrm{-mod}.
\]
\end{proposition}
In order to alleviate the notation a little, from now on we denote $\Phi_{\mathbf{s},\epsilon_{\leq t}}$ by $\Phi_{\mathbf{s}}$.  Strictly speaking, the above identity is only true after passing to the quasi-essential image of each side.  Subject to the same proviso, the following corollary follows trivially.  
\begin{corollary}
\label{Totherway}
The full subcategories $T'_{\mathbf{s}}=\Phi_{\mathbf{s}}(F_{\mathbf{s}})$ and $F'_{\mathbf{s}}=\Phi_{\mathbf{s}}(T_{\mathbf{s}}[-1])$ define a torsion structure on $\hat{J}(\mu_{\mathbf{s}}(Q,W))$, and we have
\[
\Phi_{\mathbf{s}}(\hat{J}(Q,W)\mathrm{-mod})=\hat{J}(\mu_{\mathbf{s}}(Q,W))\mathrm{-mod}^{(T'_{\mathbf{s}}[-1],F'_{\mathbf{s}})}[1].
\]
\end{corollary}

We refer to \cite{Bri07} for the definition of a Bridgeland stability condition, and we denote by $\mathbb{H}^+\subset\Cp$ the set $\{re^{i\theta}\in\Cp\mathrm{ | }\theta\in[0,\pi),r\in\mathbb{R}_{>0}\}$.  Suppose $(Q,W)$ is an algebraic QP, and we are given a  stability condition on the derived category of finite dimensional $\hat\Gamma(Q,W)$-modules whose associated heart is the abelian category of $\hat J(Q,W)$-modules.  This stability condition extends to a stability condition on the derived category of $\Gamma(Q,W)$-modules with heart the category of $J(Q,W)$-modules by precomposing the central charge map $Z:\KK(\hat J(Q,W))\rightarrow \mathbb{H}^+$ with the composition 
\[
\KK(J(Q,W)\text{-mod})\xrightarrow{\dim_{J(Q,W)\text{-mod}}}\mathbb{Z}^{Q_0}\xrightarrow{\dim^{-1}_{\hat J(Q,W)\text{-mod}}}\KK(\hat J(Q,W)\text{-mod}).
\]
Here we have used the fact that nilpotent finite-dimensional modules are all given by repeated extensions of shifted simples $S_i$ and so in particular, their class in the Grothendieck group is determined by their dimension.
\begin{proposition}\cite[Prop.4.1]{Nag13}\label{mut_is_tilt}
Let $\mathcal{C}= J(\mu_{\mathbf{s}}(Q),0)\mathrm{-mod}$, and let $W$ be a formal potential for $Q$, nondegenerate with respect to the sequence of vertices $\mathbf{s}$.  There is a Bridgeland stability condition on the derived category of $\Gamma(\mu_{\mathbf{s}}(Q),0)$-modules with central charge map \[Z:\KK(J(\mu_{\mathbf{s}}(Q),0)\mathrm{-mod})\rightarrow\mathbb{H}^+\] and an angle $\theta_{\mathbf{s}}\in [0,\pi)$ such that 
\[\Phi_{\mathbf{s}}(T_{\mathbf{s}})[-1]=\mathcal{C}_{<\theta_{\mathbf{s}}}\cap \hat{J}(\mu_{\mathbf{s}}(Q,W))\mathrm{-mod}\] 
and 
\[\Phi_{\mathbf{s}}(F_{\mathbf{s}})=\mathcal{C}_{\geq \theta_{\mathbf{s}}}\cap \hat{J}(\mu_{\mathbf{s}}(Q,W))\mathrm{-mod}.\]
\end{proposition}
Here we are using the embedding of the category of finite dimensional $\hat{J}(\mu_{\mathbf{s}}(Q,W))$-modules into the category of $J(\mu_{\mathbf{s}}(Q),0)$-modules given by the natural map $J(\mu_{\mathbf{s}}(Q),0)\rightarrow \hat{J}(\mu_{\mathbf{s}}(Q,W))$.  In the above proposition we follow the convention that $\mathcal{C}_{<\theta}$ is the full subcategory of $\mathcal{C}$ whose objects have Harder-Narasimhan filtrations featuring only semisimple objects of slope strictly less than $\theta$, and define $\mathcal{C}_{\geq \theta}$ similarly.  

\subsection{Purity and positivity}

In this section, we make the connection with quantum cluster algebras and vanishing cycles, and explain Efimov's work in \cite{Efi11}.

In order to do this, we study moduli spaces of representations for uncompleted Jacobi algebras. Note that uncompleted Jacobi algebras only make sense for \textit{algebraic} potentials. Given a quiver~$Q$, and a dimension vector $\gamma\in\mathbb{Z}^{m}_{\geq 0}$, we define the scheme
\begin{equation}
\label{qstackdef}
M_{Q,\gamma}=\bigoplus_{a\in Q_1}\Hom(\Cp^{\gamma_{t(a)}},\Cp^{\gamma_{s(a)}}),
\end{equation}
which carries a natural $\prod_{i\in Q_0} \Gl_{\Cp}(\gamma_i)$ action, where each general linear group acts by change of basis.  The quotient stack
\[
\mathcal{M}_{Q,\gamma}:=M_{Q,\gamma}\big/\prod_{i\in Q_0}\Gl_{\Cp}(\gamma_i)
\]
is the moduli stack of \textit{right} $J(Q,0)$-modules with dimension vector $\gamma$. 

For an algebraic $W$, $\hat{J}(Q,W)\mathrm{-mod}$ naturally embeds into $J(Q,W)\mathrm{-mod}$ as the full subcategory of nilpotent modules.  In turn, there is a natural embedding of $J(Q,W)\mathrm{-mod}$ into $J(Q,0)\mathrm{-mod}$.    The second of these embeddings is described stack theoretically as the embedding of the stack theoretic critical locus of $\tr(W)_{\gamma}$.  The first of these embeddings has a slightly more involved stack theoretic counterpart, to be studied in Section~\ref{subsec:families} below.
\smallbreak
\begin{remark}
The algebraicity of $W$ also ensures that we get an algebraic function $\tr(W)_{\gamma}$ on the stack $\mathcal{M}_{Q,\gamma}$. 
It is to this function that we will apply our results on vanishing cycles under suitable further hypotheses. 
\end{remark}
\begin{assumption}
\label{mut_alg}
Given a quiver $Q$ and a sequence of vertices $\mathbf{s}=(s_1,\ldots, s_t)$, we assume $W$ is a potential for $Q$ which is nondegenerate with respect to $\mathbf{s}$, such that $\mu_{\mathbf{s}}(Q,W)$ can be chosen to be an algebraic QP.  We denote by $W^{\diamondsuit}$ such a choice of algebraic potential on $\mu_{\mathbf{s}}(Q)$.  
\end{assumption}
We prefer the notation $W^{\diamondsuit}$ over, say, $\mu_{\mathbf{s}}(W)$, since the mutated potential is only defined up to the equivalence relation defined by formal automorphisms, and we wish to make it clear that in this instance we have chosen a particular type of element of the equivalence class - an algebraic potential.  A potential $W$ for which we can pick an algebraic member of the equivalence class of $\mu_{\mathbf{s}}(W)$ always exists: take an algebraic potential $W^{\diamondsuit}$ for $\mu_{\mathbf{s}}(Q)$ that is nondegenerate with respect to the reverse sequence $\overleftarrow{\mathbf{s}}=(s_t,\ldots,s_1)$, which exists by Corollary~\ref{formaltoalg} and existence of nondegenerate potentials \cite[Cor.7.4]{DWZ08}, and set $(Q,W)=\mu_{\overleftarrow{\mathbf{s}}}(\mu_{\mathbf{s}}(Q),W^{\diamondsuit})$.  Again we use the fact that mutation defines an involution on (equivalence classes of) QPs \cite[Thm.5.7]{DWZ08}.
\smallbreak

\begin{remark}
In the event that $(\mu_{\mathbf{s}}(Q),W^{\diamondsuit})$ is a graded algebraic QP, the function $\tr(W^{\diamondsuit})_{\gamma}$ is $\Cp^*$-equivariant after giving $\Hom(\Cp^{\gamma_{t(a)}},\Cp^{\gamma_{s(a)}})$ the weight $|a|$-action, and giving the target $\Cp$ the weight $|W|$-action.
\end{remark}
We introduce a partial ordering on dimension vectors $\gamma\in\mathbb{Z}_{\geq 0}^{Q_0}$ by defining $\gamma'<\gamma$ if $\gamma'_i\leq \gamma_i$ for all $i$, and $\gamma'_i<\gamma_i$ for at least one $i$.  For each dimension vector $\gamma'\leq\gamma$ we define the subspace
\[
M_{\mu_{\mathbf{s}}(Q),\gamma,\gamma'}\subset M_{\mu_{\mathbf{s}}(Q),\gamma}\times\prod_{i\in Q_0} \Gr(\gamma'_i,\gamma_i)
\]
of pairs of a representation $M$ and a flag preserved by $M$. Next, recalling the angle $\theta_{\mathbf{s}}$ from Proposition~\ref{mut_is_tilt}, define $M_{\mu_{\mathbf{s}}(Q),\gamma,<\theta_{\mathbf{s}}}$ to be the complement of the union of the images of the maps $M_{\mu_{\mathbf{s}}(Q),\gamma,\gamma'}\rightarrow M_{\mu_{\mathbf{s}}(Q),\gamma}$ for $\gamma'\leq\gamma$ such that $\arg(Z(\gamma'))\geq \theta_{\mathbf{s}}$.  Since these maps are proper, $M_{\mu_{\mathbf{s}}(Q),\gamma,<\theta_{\mathbf{s}}}$ is an open subscheme of $M_{\mu_{\mathbf{s}}(Q),\gamma}$.  
\smallbreak
\smallbreak
Following \cite{Efi11}, given a nonzero dimension vector $\nn\in \mathbb{Z}^{Q_0}_{\geq 0}$ we define the space of framed representations
\[
M^{\fr}_{\nn,\mathbf{s},\gamma}:=\{(E\in M_{\mu_{\mathbf{s}}(Q),\gamma,< \theta_{\mathbf{s}}},u:J(\mu_{\mathbf{s}}(Q),0)_{\nn}\rightarrow E) \}
\]
and within that, the space of stable framed representations
\[
M^{\sfr}_{\nn,\mathbf{s},\gamma}:=\{(E\in M_{\mu_{\mathbf{s}}(Q),\gamma,< \theta_{\mathbf{s}}},u:J(\mu_{\mathbf{s}}(Q),0)_{\nn}\rightarrow E) \mathrm{ | }\coker(u)\in \mathcal{C}_{\geq \theta_{\mathbf{s}}}\},
\]
where for a general algebraic QP $(Q',W')$, we define $J(Q',W')_{\nn}:=\oplus_{i\in Q'_0}(e_i\cdot J(Q',W'))^{\oplus \nn_i}$.  The group $G_\gamma=\prod_{i\in Q_0}\Gl_{\Cp}(\gamma_i)$ acts on $M^{\sfr}_{\nn,\mathbf{s},\gamma}$ via conjugation on $E$ and postcomposition on~$u$, and we define the stack theoretic quotient
\[
\mathcal{M}^{\sfr}_{\nn,\mathbf{s},\gamma}:=M^{\sfr}_{\nn,\mathbf{s},\gamma}\big/G_\gamma.
\]
We include a proof of the following remark of \cite{Efi11} for completeness.
\begin{proposition}
\label{sfr_smooth}
The stack $\mathcal{M}^{\sfr}_{\nn,\mathbf{s},\gamma}$ is a smooth quasiprojective variety.  
\end{proposition}
\begin{proof} We first prove that $M^{\sfr}_{\nn,\mathbf{s},\gamma}$ is itself smooth. For $\gamma'<\gamma$ define 
\[
M^{\fr}_{\nn,\mathbf{s},\gamma,\gamma'}:=\{((E'\subset E)\in M_{\mu_{\mathbf{s}}(Q),\gamma,\gamma'},u:J(\mu_{\mathbf{s}}(Q),0)_{\nn}\rightarrow E')|E\in M_{\mu_{\mathbf{s}}(Q),\gamma,<\theta_{\mathbf{s}}} \}.
\]
The projection to $M^{\fr}_{\nn,\mathbf{s},\gamma}$ is a proper map, and $M^{\sfr}_{\nn,\mathbf{s},\gamma}$ is the complement to the images of the finite collection of proper maps $M^{\fr}_{\nn,\mathbf{s},\gamma,\gamma'}\rightarrow M^{\fr}_{\nn,\mathbf{s},\gamma}$ for $\arg(Z(\gamma-\gamma'))<\theta_{\mathbf{s}}$, so it is open in $M^{\fr}_{\nn,\mathbf{s},\gamma}$.  The space $M^{\fr}_{\nn,\mathbf{s},\gamma}$ is itself an affine fibration over the scheme $M_{\mu_{\mathbf{s}}(Q),\gamma,<\theta_{\mathbf{s}}}$ which is in turn an open subscheme of the smooth scheme $M_{\mu_{\mathbf{s}}(Q),\gamma}$, so it is smooth.  So $M^{\sfr}_{\nn,\mathbf{s},\gamma}$ is an open subscheme of a smooth variety and it is smooth.

Next, by the standard argument recalled in \cite[Prop.3.7]{Efi11}, the $G_\gamma$-action on $M^{\sfr}_{\nn,\mathbf{s},\gamma}$ 
is free.

We will now proceed to show that $\mathcal{M}^{\sfr}_{\nn,\mathbf{s},\gamma} = M^{\sfr}_{\nn,\mathbf{s},\gamma}/G_\gamma$ is a 
GIT quotient; this will prove that $\mathcal{M}^{\sfr}_{\nn,\mathbf{s},\gamma}$ is indeed smooth and quasiprojective as
claimed. To show this, following the argument of \cite[Prop.3.7]{Efi11}, introduce the auxiliary quiver $\tilde Q$ which 
has vertex set $\tilde Q_0=Q_0\cup\{v\}$, the same arrows as $\mu_{\mathbf{s}}(Q)$, and additionally $\nn_i$ arrows 
from $i$ to $v$ for each $i\in Q_0$. Then any object in $M^{\fr}_{\nn,\mathbf{s},\gamma}$ is equivalently 
a representation of the quiver $\tilde Q$ with dimension vector $(\gamma, 1)$.
Extend also the central charge map $Z:\mathbb{Z}^{Q_0}\rightarrow\mathbb{H}^+$
to a central charge map $\tilde Z:\mathbb{Z}^{\tilde Q_0}\rightarrow\mathbb{H}^+$ by defining 
$\tilde Z(1_v)=t\exp(\alpha \sqrt{-1})$, where 
$1_v\in\mathbb{Z}^{\tilde Q_0}$ denotes the lattice generator corresponding to the new vertex~$v$, 
$\alpha<\theta_{\mathbf{s}}$ is an angle, and $t$ is a positive real number. Then a standard argument shows 
that for $\alpha$ sufficiently close to $\theta_{\mathbf{s}}$ and $t$ sufficiently large, we have an equality 
\[ M^{\tilde Z}_{(\gamma,1)}=  M^{\sfr}_{\nn,\mathbf{s},\gamma}\]
between the $\tilde Z$-stable representations of $\tilde Q$ with dimension vector $(\gamma, 1)$ and the stable framed 
representations of $\mu_{\mathbf{s}}(Q)$ defined above. We now conclude by Lemma~\ref{br=k} below, explained to us 
by Efimov. 
\end{proof}
\begin{lemma} Let $Q$ be an arbitrary quiver. Fix a Bridgeland stability condition on the derived category of
${\mathbb C}Q$-modules with heart ${\mathbb C}Q\text{-mod}$ and central charge map 
$Z\colon {\mathbb Z}^{Q_O}\to \mathbb{H}^+$.
Fix also a dimension vector $\gamma\in {\mathbb N}^{Q_0}$. Then there is a stability parameter in the sense of King~\cite{king}
\[\theta \in \left\{ \theta(\gamma)=0\right\}\subset {\mathbb Q}^{Q_0} \subseteq \Hom\left(K_0({\mathbb C}Q\text{\rm -mod}), {\mathbb Q}\right)\]
such that a representation of $Q$ with dimension vector $\gamma$ is $Z$-stable if and only if it is $\theta$-stable.
\label{br=k}
\end{lemma}
\begin{proof} Define the map $\theta$ on basis vectors $1_i\in{\mathbb Z}^{Q_0}$ corresponding to vertices $i\in Q_0$ by
\[ \theta(1_i) = {\re Z(1_i)} - \frac{\re Z(\gamma)}{\im Z(\gamma)} \im Z(1_i)
\]
and extend by linearity. Then it is immediate that $\theta(\gamma)=0$, and 
for any dimension vector $\gamma'<\gamma$ (in the partial order defined above), 
$\arg Z(\gamma')<\arg Z(\gamma)$ if and only if $\theta(\gamma')>\theta(\gamma)=0$. 
\end{proof}

In the case when~$Q$ is a graded quiver, $\mathcal{M}^{\sfr}_{\nn,\mathbf{s},\gamma}$ carries a natural $\Cp^*$-action, via the induced action on the linear maps defining $E$ and the action on $u$ which leaves the $u(e_i)$ invariant, where the $e_i$ are the length zero paths in $J(\mu_{\mathbf{s}}(Q),0)$.  
\smallbreak
Assume that $(Q,W)$ satisfies Assumption~\ref{mut_alg}, with respect to the sequence of vertices $\mathbf{s}$.  The function $\tr(W^{\diamondsuit})_{\nn,\gamma}$ is defined on $\mathcal{M}^{\sfr}_{\nn,\mathbf{s},\gamma}$  by setting $\tr(W^{\diamondsuit})_{\nn,\gamma}=\tr(W^{\diamondsuit})_{\gamma}\circ \pi_{\nn,\gamma}$, where $\pi_{\nn,\gamma}:\mathcal{M}^{\sfr}_{\nn,\mathbf{s},\gamma}\rightarrow\mathcal{M}_{\mu_{\mathbf{s}}(Q),\gamma}$ is the natural projection. By the proof of Proposition~\ref{sfr_smooth}, we have that 
\begin{align}
\label{sfr_is_crit}
\mathcal{M}^{\sfr,W^{\diamondsuit}}_{\nn,\mathbf{s},\gamma}:=&\{(E\in M_{\mu_{\mathbf{s}}(Q),\gamma,< \theta_{\mathbf{s}}}\cap J(\mu_{\mathbf{s}}(Q,W))\mathrm{-mod},u:J(\mu_{\mathbf{s}}(Q,W))_{\nn}\rightarrow E)\mathrm{ | }\\ \notag&\coker(u)\in \mathcal{C}_{\geq \theta}\}\Big/\prod_{i\in Q_0} \Gl_{\Cp}(\gamma_{i})
\\ \notag=&\crit(\tr(W^{\diamondsuit})_{\nn,\gamma}).
\end{align}
Taking the reduced zero locus of the functions $\tr(l)$, where $l$ ranges over the set of linear combinations of nontrivial cycles in $\mu_{\mathbf{s}}(Q)$, we define the closed subscheme 
$$\mathcal{M}^{\sfr,\Sp,W^{\diamondsuit}}_{\nn,\mathbf{s},\gamma}\subset \mathcal{M}^{\sfr,W^{\diamondsuit}}_{\nn,\mathbf{s},\gamma},$$ 
the set-theoretic locus of nilpotent modules. Note that by a theorem of Le Bruyn and Procesi~\cite[Thm.1]{LBP90}, if $W^{\diamondsuit}=0$ then the ideal generated by all the $\tr(l)$ is in fact reduced. 

Finally, we come to the precise statement that links Hodge theory with quantum cluster transformations.  Given $J$ a complex of mixed Hodge structures, we define the weight polynomial $\chi_W([J],t):=\sum_{i,m\in\mathbb{Z}}(-1)^m \dim \gr_W^i(J^m)t^i$ as usual.  For $H$ a complex of $\mu_n$-equivariant mixed Hodge structures, we define $\chi'_W(H,t)=\chi_W(H_1,t)+t\chi_W(H_{\neq 1},t)$, where $H_1$ is the $\mu_n$ invariant part of $H$, and $H_{\neq 1}$ is its unique $\mu_n$-equivariant complement in $H$.  Then the following result is proved in \cite{Efi11}.
\begin{theorem}\cite[Thm.5.3]{Efi11}
\label{hodgetocluster}
Assume that $(Q,M)$ is the usual initial seed, with associated quantum cluster $(X_1, \ldots ,X_n)$, and let $W$ be a potential for $Q$ satisfying Assumption~\ref{mut_alg} with respect to the sequence of vertices $\mathbf{s}$. Then the mutated toric frame $\mu_{\mathbf{s}}(M)$ has the following expression
\begin{align}
\label{efthm}
&\mu_{\mathbf{s}}(M)(\nn)= \\ \notag &X^{[\hat{\Gamma}(\mu_s(Q,W))_{\nn}]}\sum_{\gamma\in\mathbb{Z}_{\geq 0}^n}\chi'_W([\Ho^*(\mathcal{M}^{\sfr,\Sp,W^{\diamondsuit}}_{\nn,\mathbf{s},\gamma},\phi_{\tr(W^{\diamondsuit})_{\nn,\gamma}})]),-q^{-1/2})q^{-\chi_{\mu_{\mathbf{s}}(Q)}(\gamma,\gamma)/2}X^{\iota(-\Phi_{\mathbf{s}}^{-1}(\gamma))}.
\end{align}
\end{theorem}

We conclude this section with a series of long remarks clarifying aspects of the above statement.  With the exception of Proposition \ref{intpowers}, taken from \cite{Efi11}, they can be safely skipped on a first reading.

\begin{remark}
In this remark, we define the exponents of (\ref{efthm}) -- although of course for considering positivity and the Lefschetz property these definitions are strictly speaking irrelevant.  First, for arbitrary QP $(Q',W')$, we define 
\[\hat{\Gamma}(Q',W')_{\nn}:=\oplus_{i\leq n}(e_i\cdot \hat{\Gamma}(Q',W'))^{\oplus \nn_i}. \]We fix $\KK(\Perf(\hat{\Gamma}(Q,W)\mathrm{-Mod})\cong\mathbb{Z}^n$ via the map on positive vectors \[\tau:\nn\mapsto\hat\Gamma(Q,W)_{\nn}.\]   We define $X^{[\hat{\Gamma}(\mu_{\mathbf{s}}(Q,W))_{\nn}]}:=X^{\tau^{-1}\KK(\Phi_{\mathbf{s}})^{-1}([\hat{\Gamma}(\mu_{\mathbf{s}}(Q,W))_{\nn}])}$.   The map $\iota$ is defined by applying $\KK$ to the inclusion $\hat{\Gamma}(Q,W)\mathrm{-mod}\rightarrow \Perf(\hat{\Gamma}(Q,W)\mathrm{-Mod})$ and composing with $\tau^{-1}$.  Finally, for an arbitrary quiver $Q'$, the Ringel form $\chi_{Q'}:\mathbb{Z}^{Q'_0}\times \mathbb{Z}^{Q'_0}\rightarrow \mathbb{Z}$ is defined by \[\chi_{Q'}(\mathbf{n},\mathbf{m})=\sum_{i\in Q'_0}\mathbf{n}_i\mathbf{m}_i-\sum_{a\in Q'_1}\mathbf{n}_{s(a)}\mathbf{m}_{t(a)}.\]
\end{remark}

\begin{remark}
The form of (\ref{efthm}) is slightly different from what appears in~\cite[Thm.5.3]{Efi11}.  Above we consider (modified) weight polynomials of usual mixed Hodge structures.  Efimov considers instead classes in the Grothendieck group of monodromic mixed Hodge structures, as considered in \cite{KS08}, which have their own version of the weight polynomial.  These are the natural coefficients for generating series in the cohomological Hall algebra, while our weight polynomials should be seen as realisations of elements in the naive Grothendieck ring of $\hat{\mu}$-equivariant motives (motivic vanishing cycles).  The passage between generating series in the motivic Hall algebra and the cohomological Hall algebra is discussed in Section 7.10 of \cite{KS10} - it is realised by a map on coefficients that pulls back the weight polynomial for monodromic mixed Hodge structures to $\chi'_W(t^{-1})$.  The minus sign that appears in the exponent in the weight polynomial on the right hand side of (\ref{efthm}) is a result of the duality functor appearing in~\cite[Thm.5.3]{Efi11}.  Note that the definition of the 
modified weight polynomial is just the specialization at $z_1=z_2=q^{1/2}$ of a suitable equivariant Hodge polynomial (as in \cite[p.69]{KS08}), which defines a ring homomorphism on the naive Grothendieck ring of $\hat{\mu}$-equivariant motives with a convolution product, the coefficients of the motivic Hall algebra.
\end{remark}

\begin{remark}
In this remark, we comment further on the minus sign that appears as the coefficient of $q^{-1/2}$ in the weight polynomial.
The source of this minus sign is the definition of the integration map to the motivic quantum torus of \cite{KS08}, defined in [ibid, Sec.6.3] -- that is, we take the twisted weight polynomial obtained after substituting $q^{1/2}\mapsto -q^{1/2}$.  We recall from [ibid] that the weight polynomial of the square root of the Tate motive is $-q^{1/2}$, and its Euler characteristic is $-1$.  This choice is made so that, amongst other things, the Euler characteristic of the virtual motive $\mathbb{L}^{-\dim(X)/2}[X]$ of a smooth scheme $X$ is the same as its weighted Euler characteristic $\chi(X,\nu_X)$, where $\nu_X$ is Behrend's microlocal function \cite{Beh09}.  By picking this sign, the expression (\ref{efthm}) becomes precisely the result of applying the Kontsevich-Soibelman integration map to an elementary identity in the motivic Hall algebra of \cite{KS08}, see \cite[Sec.7.2]{Nag13} for an exposition of this.  Note that because we take this twisted weight polynomial, the Euler characteristic specialization is $q^{1/2}=-1$, while the specialization that recovers classical from quantum cluster mutation is $q^{1/2}=1$.  So in fact, in this setup, positivity follows from purity, without any statements regarding vanishing of odd cohomology.  However, in order to justify the part of the Lefschetz property (Conjecture~\ref{LefProperty}) regarding the parity of $k$, we will still need the following proposition.
\end{remark}
\begin{proposition}\cite[(5.5) of Thm 5.3]{Efi11}
\label{intpowers}
The weight polynomial \[\chi'_W([\Ho^*(\mathcal{M}^{\sfr,\Sp,W^{\diamondsuit}}_{\nn,\mathbf{s},\gamma},\phi_{\tr(W^{\diamondsuit})_{\nn,\gamma}})]),-q^{-1/2})\] appearing in Theorem~\ref{hodgetocluster} is a Laurent polynomial in $q$.
\end{proposition}

\begin{remark} 
Efimov's work relies on the two Kontsevich--Soibelman papers~\cite{KS08, KS10}, some aspects of 
which remain conjectural. For us, Efimov's main result is~\cite[Thm.5.3]{Efi11}, which
leads to the expression~\eqref{efthm} in terms of a weight polynomial. 
In this remark, we discuss precisely what are the necessary ingredients to arrive at this result. 

Efimov's work is based on the wall-crossing technology of Kontsevich--Soibelman~\cite{KS08, KS10}, but 
only after specializing to the Grothendieck ring of exponential mixed Hodge structures.  
The main technical ingredients required for these results are the Thom--Sebastiani 
Theorem and the ``integral identity'' \cite[Conj.4]{KS08}.  Both statements can be refined, either in terms 
of complexes of mixed Hodge structures (without passing to the $K$-ring) or in terms of the naive Grothendieck ring 
of $\hat{\mu}$-equivariant motives.  The status of these refinements is slightly different for each statement.
The Thom-Sebastiani theorem is firmly established in the motivic context~\cite[Thm.5.18]{GLM}, 
(see also \cite[Thm.5.2.2]{DL99}), meaning where we work with the naive Grothendieck ring 
of $\hat{\mu}$-equivariant motives. Its proof in the Hodge-theoretic context remains in preprint form for now~\cite{S3}.
The integral identity is proved in the Hodge-theoretic case (for critical cohomology, the case relevant for us)
in~\cite[Sect.7.1]{KS10}, while the general motivic case has only been claimed recently~\cite{LQT}. 
In either case, once we pass to the Hodge-theoretic Grothendieck ring, both statements are known; i.e., one can use the motivic 
side whenever one needs Thom--Sebastiani, and the Hodge-theoretic side whenever one needs the integral identity.

A further aspect of Efimov's work, contained in~\cite[Appendix]{Efi11}, is
the need to associate Donaldson--Thomas type invariants to quivers with
formal potential, using, as above, a mutation equivalence to a 
quiver with algebraic potential. 
As the discussion provided there is rather brief, a more detailed
exposition of these ideas will be presented in~\cite{D}.

\label{remark:statusofKS}\end{remark}
\subsection{Families of nilpotent modules}
\label{subsec:families}

We now assume we are in the situation of Assumption~\ref{mut_alg}.  In this section, we study 
the locus of nilpotent representations of the algebraic QP $(\mu_{\mathbf{s}}(Q), W^{\diamondsuit})$; the 
main goal is to show that nilpotence is essentially an open and closed condition for stable framed representations.

Associated to the quiver $\mu_{\mathbf{s}}(Q)$, we have the smooth scheme 
$\mathcal{M}^{\sfr}_{\mathbf{n},\mathbf{s},\gamma}$ equipped with the function
$\tr(W^{\diamondsuit})_{\nn,\gamma}$
whose critical locus is 
$\mathcal{M}^{\sfr,W^{\diamondsuit}}_{\mathbf{n},\mathbf{s},\gamma}$. Inside the latter, we have
the closed subscheme $\mathcal{M}^{\sfr,\Sp,W^{\diamondsuit}}_{\nn,\mathbf{s},\gamma}$, the 
set-theoretic locus of nilpotent modules.
In order to apply Corollary~\ref{cor:pure}  to the result of Theorem~\ref{hodgetocluster}, 
we need to show the following:
\begin{proposition}\label{unionofcomponents}
The subvariety $\mathcal{M}^{\sfr,\Sp,W^{\diamondsuit}}_{\mathbf{n},\mathbf{s},\gamma}$ of $\ \mathcal{M}^{\sfr,W^{\diamondsuit}}_{\mathbf{n},\mathbf{s},\gamma}$ is projective, and is a union of connected components of the reduced support of $\mathcal{M}^{\sfr,W^{\diamondsuit}}_{\mathbf{n},\mathbf{s},\gamma}$.
\end{proposition}

We will prove this proposition using the following lemma.  In what follows, let 
 $X$ be a scheme of finite type and $Y \subset X$ a closed subscheme, and let
$$Y = Y_1 \subset Y_2 \subset Y_3 \subset \dots \subset X$$
denote the chain of infinitesimal thickenings of $Y$ inside $X$, defined by powers of the ideal sheaf $\mathcal{I}_{Y/X}$.
We can define a covariant functor $\colim Y_{*}$ from commutative $\C$-algebras to sets by the prescription
$$A \mapsto \colim Y_{*}(A).$$

\begin{lemma}\label{indvarieties}
Suppose the functor $\colim Y_*$ is represented by a scheme $Z$ of finite type.
Then the sequence $\{Y_{*}\}$ of subschemes stabilizes to $Z \subset X$
and the reduced support of $Z$ is a union of connected components of the reduced support of $X$.
\end{lemma}
\begin{proof}
We have a map of ind-schemes $g:  Z \rightarrow \colim Y_*$ which induces the equivalence of functors. Since $Z$ is finite type, $g$ is induced by 
a map $g: Z \rightarrow Y_N$ for some $N$.  As a result, for $m > N$, the natural map $Y_m \rightarrow \colim Y_*$ factors through $Y_N$.  This is only possible if the inclusion $Y_N \hookrightarrow Y_m$ is an isomorphism for all $m >N$.  Therefore, we have a natural identification $Z = \lim Y_* = Y_N$, so that $Y_N = Z$ as schemes. For the second claim, note that 
the ideal sheaf $\mathcal{I}_{Z/X}$ equals its own square, so must locally be zero 
or contain a unit by Nakayama's lemma.  Therefore $Z$ is an open and closed subscheme of $X$ which 
implies the result.
\end{proof}

We will apply this lemma to the inclusion $\mathcal{M}^{\sfr,\Sp, W^{\diamondsuit}}_{\mathbf{n},\mathbf{s},\gamma} \subset \mathcal{M}^{\sfr, W^{\diamondsuit}}_{\mathbf{n},\mathbf{s},\gamma}.$
In order to show $$\colim\left(\mathcal{M}^{\sfr,\Sp, W^{\diamondsuit}}_{\mathbf{n},\mathbf{s},\gamma}\right)_{\ast}$$ is represented by a projective scheme, we show that it classifies families of framed nilpotent modules and use the derived equivalence $\Phi_{\mathbf{s}}$ to identify it with a quiver Grassmannian associated to $(Q,W)$.

\begin{definition}\label{nilpotentmodule}
Given a commutative $\C$-algebra $A$, a nilpotent $\hat{J}(\mu_{\mathbf{s}}(Q),W^{\diamondsuit})$-module over $A$ is a finite projective $A$-module $F$
equipped with an action of $\hat{J}(\mu_{\mathbf{s}}(Q),W^{\diamondsuit})/\hat{J}(\mu_{\mathbf{s}}(Q),W^{\diamondsuit})_{\geq m}$ by $A$-endomorphisms, for some $m \geq 1$; here we use the trivial (path length) grading. 
\end{definition}
\begin{remark}\label{geometricpoints}
When $A$ is finitely generated, a $J(\mu_{\mathbf{s}}(Q),W^{\diamondsuit})$-module $F$ over $A$ is nilpotent if and only if, for each geometric point
$\spec K \rightarrow \spec A$, the base change $F \otimes K$ is nilpotent over~$K$.  Indeed, one can bound the order of nilpotence of $F$ by 
the maximum order of nilpotence of $F$ over generic points of $\spec A$ multiplied by the the order of nilpotence of the nilradical of $A$.
\end{remark}
\begin{definition}\label{framednilpotentmodule}
Fix a $\C$-algebra $A$ and vectors $\mathbf{n}$ and $\gamma$. We define a stable, framed nilpotent $\hat{J}(\mu_{\mathbf{s}}(Q),W^{\diamondsuit})$-module over $A$ to be a nilpotent module $F$ over $A$ equipped with a morphism
$$u: \hat{J}(\mu_{\mathbf{s}}(Q),W^{\diamondsuit})_{\mathbf{n}}\otimes A \rightarrow F$$
such that, after restriction to any geometric point $\spec K \rightarrow \spec A$,
\begin{itemize}
\item[(i)] $F \otimes K$ has dimension vector $\gamma$ and lies in $\mathcal{C}_{< \theta_{\mathbf{s}}}$;
\item[(ii)] $\mathrm{coker}(u) \otimes K$ lies in $\mathcal{C}_{\geq \theta_{\mathbf{s}}}$;
\end{itemize}
in other words, over each geometric point of $\spec A$, we have a $K$-point of $\mathcal{M}^{\sfr,\Sp,W^{\diamondsuit}}_{\mathbf{n},\mathbf{s},\gamma}$.
\end{definition}
 
Let $\mathcal{M}^{\sfr,\nilp, W^{\diamondsuit}}_{\nn,\mathbf{s},\gamma}\!\!(A)$ denote the set of stable framed nilpotent modules over $A$.
The functor $\mathcal{M}^{\sfr,\nilp, W^{\diamondsuit}}_{\nn,\mathbf{s},\gamma}$ defined in this way commutes with directed colimits, since framed nilpotent modules are determined by the action on finitely many generators and the constraint on slopes is an open condition.

\begin{lemma}\label{familiesofmodules}
We have a natural identification 
$$\mathcal{M}^{\sfr,\nilp, W^{\diamondsuit}}_{\nn,\mathbf{s},\gamma} \stackrel\sim\longrightarrow \colim\left(\mathcal{M}^{\sfr,\Sp,W^{\diamondsuit}}_{\mathbf{n},\mathbf{s},\gamma}\right)_{\ast}.$$
\end{lemma}
\begin{proof}
Both sides commute with directed colimits, so it suffices to construct this identification on finitely generated rings $A$.  By Remark~\ref{geometricpoints}, each set 
consists of $A$-points of 
$\mathcal{M}^{\sfr,W^{\diamondsuit}}_{\mathbf{n},\mathbf{s},\gamma}$
whose geometric points lie in the closed subvariety
$\mathcal{M}^{\sfr,\Sp,W^{\diamondsuit}}_{\mathbf{n},\mathbf{s},\gamma}$.
\end{proof}

Consider again the equivalence of dg-categories $\Phi_{\mathbf{s}}$.  
By construction \cite{KY11}, there exists a bimodule $S$, with left $\hat{\Gamma}(\mu_{\mathbf{s}}(Q),W^{\diamondsuit})$ action and right $\hat{\Gamma}(Q,W)$ action, such that $\Phi_{\mathbf{s}}$ is the functor 
\[\Hom_{\hat{\Gamma}(Q,W)}(S,-):\hat{\Gamma}(Q,W)\text{-mod}\rightarrow \hat{\Gamma}(\mu_{\mathbf{s}}(Q),W^{\diamondsuit})\text{-mod}.\]  
Consider the dg module
$$\hat{\Gamma}(\mu_{\mathbf{s}}(Q),W^{\diamondsuit})_{\nn}:=\oplus_{i\leq n}(e_i\cdot \hat{\Gamma}(\mu_{\mathbf{s}}(Q),W^{\diamondsuit}))^{\oplus \nn_i}$$
and its preimage
$$P_{\nn} = \Phi_{\mathbf{s}}^{-1} \left(\hat{\Gamma}(\mu_{\mathbf{s}}(Q),W^{\diamondsuit})_{\nn}\right).$$
By \cite[Prop.2.18]{plamondon}, $P_{\nn}$ is quasi-isomorphic to a module which, as a graded
module, is a sum of a finite sum of modules $e_i\cdot\hat{\Gamma}(Q,W)$ and a
finite sum of modules $e_i\cdot\hat{\Gamma}(Q,W)[-1]$. In particular, $P_{\nn}$ is concentrated in degrees less 
than or equal to 1, and we have a natural truncation morphism
$$P_{\nn}[1] \rightarrow \Ho^1(P_{\nn}).$$
As explained in~\cite[Cor.4.11]{Efi11}, $\Ho^1(P_{\nn})$ is 
a finite-dimensional $\hat{J}(Q,W)$-module. We will need the following stronger statement below. 

\begin{lemma} We have $\Ho^1(P_{\mathbf{n}})\in T_{\mathbf{s}}$. \label{lem_H1_in_torsion}
\end{lemma}
\begin{proof} Consider the distinguished triangle
\[
K\rightarrow P_{\mathbf{n}}\rightarrow \Ho^1(P_{\mathbf{n}})[-1].
\]
The module $\Ho^1(P_{\mathbf{n}})[-1]$ is perfect, since it is
finite-dimensional as mentioned above.  The module $K$ has cohomology supported in degrees less than or equal
to zero, and is also perfect, as $P_{\mathbf{n}}$ is by
definition.  From
\cite[Prop.2.18]{plamondon} we deduce that $\Phi_{\mathbf{s}}$ sends
summands $e_i\cdot \hat{\Gamma}(Q,W)$ to extensions of summands
$e_i\cdot \hat{\Gamma}(\mu_{\mathbf{s}}(Q),W^{\diamondsuit})[t]$ for
$t=0,1$. In particular, for any $l\in\mathbb{Z}$ and for $M$ a perfect
$\hat{\Gamma}(Q,W)$-module, , $\Ho^{\geq l}(M)=0$ implies $\Ho^{\geq l}(\Phi_{\mathbf{s}}(M))=0$.  It follows that $\Ho^{\geq 1}(\Phi_{\mathbf{s}}(K))$ vanishes, as does $\Ho^{\geq 1}(\Phi_{\mathbf{s}}(P_{\mathbf{n}}))$ by definition, and so from the long exact sequence in cohomology, $\Ho^1(\Phi_{\mathbf{s}}(\Ho^1(P_{\mathbf{n}})[-1]))=0=\Ho^0(\Phi_{\mathbf{s}}(\Ho^1(P_{\mathbf{n}})))$.  Since $\Ho^1(P_{\mathbf{n}})$ is a finite-dimensional $\hat{J}(Q,W)$-module, we deduce from Corollary \ref{Totherway} that $\Ho^1(P_{\mathbf{n}})\in T_{\mathbf{s}}$.
\end{proof}

If we consider the quiver Grassmannian
$$\mathrm{Grass}(\Ho^1(P_{\nn}),-\Phi_{\mathbf{s}}^{-1}(\gamma))$$
of quotient submodules, this is naturally a closed subscheme of the usual Grassmannian of $\Ho^1(P_{\nn})$ viewed as a vector space. 

\begin{proposition}\label{equivalenceofstacks}
There is an equivalence of functors 
$$\PHI_{\mathbf{s}}:  \mathrm{Grass}(\Ho^1(P_{\nn}),-\Phi_{\mathbf{s}}^{-1}(\gamma)) \stackrel\sim\longrightarrow \mathcal{M}^{\sfr,\nilp, W^{\diamondsuit}}_{\nn,\mathbf{s},\gamma}.$$
In particular, 
$\mathcal{M}^{\sfr,\nilp, W^{\diamondsuit}}_{\nn,\mathbf{s},\gamma}$ is represented by a projective scheme.
\end{proposition}

\begin{proof}
It suffices to construct the equivalence on finitely generated $\C$-algebras $A$.
When $A = \C$, this is shown in \cite[Thm.5.3]{Efi11}. 

Consider the general case.
An $A$-point of $\mathrm{Grass}(\Ho^1(P_{\nn}),-\Phi_{\mathbf{s}}^{-1}(\gamma))$ consists of 
a $\hat{J}(Q,W)\otimes A$-module $E$ and a surjection
$$v: \Ho^1(P_{\nn})\otimes A \rightarrow E \rightarrow 0,$$
such that $E$ is finite and projective over $A$, with dimension vector $-\Phi_{\mathbf{s}}^{-1}(\gamma)$.

To define $\PHI_{\mathbf{s}}$, we compose $v$ with the truncation $P_{\nn}[1] \otimes A \rightarrow \Ho^1(P_{\nn})\otimes A$
to obtain a morphism of $\hat{\Gamma}(Q,W)\otimes A$-modules $v': P_{\nn}[1]\otimes A \rightarrow E$.
We then apply the functor $\Hom_{\hat{\Gamma}(Q,W)}(S,-)$
and pass to cohomology in degree $-1$.

When we apply $\Phi_{\mathbf{s}}$ to $P_{\nn}[1]\otimes A$ and pass to cohomology, we recover $\hat{J}(\mu_{\mathbf{s}}(Q),W^{\diamondsuit})_{\mathbf{n}}\otimes A$.
For the module $E$, we will now show that
$\Phi_{\mathbf{s}}(E) = \Hom_{\hat{\Gamma}(Q,W)}(S, E)$ 
has cohomology supported in degree $-1$, which is finite and projective as an $A$-module.  When $A = \C$, this follows from the definition of $\Phi_{\mathbf{s}}$.

Recall that $\Phi_{\mathbf{s}}$
induces a triangulated equivalence $\Perf(\hat{\Gamma}(Q,W))\rightarrow\Perf(\hat{\Gamma}(\mu_{\mathbf{s}}(Q,W)))$. We deduce that $S$ is perfect as a $\hat{\Gamma}(Q,W)$-module, since the quasi-inverse to $\Phi_{\mathbf{s}}$ is given by the functor $-\otimes_{\hat{\Gamma}(\mu_{\mathbf{s}}(Q,W)}S$, and $\hat{\Gamma}(\mu_{\mathbf{s}}(Q,W))$ is obviously a perfect $\hat{\Gamma}(\mu_{\mathbf{s}}(Q,W))$-module, so $\hat{\Gamma}(\mu_{\mathbf{s}}(Q,W))\otimes_{\hat{\Gamma}(\mu_{\mathbf{s}}(Q,W))}S$ is a perfect $\hat{\Gamma}(Q,W)$-module.  This module is isomorphic to $S$ in the derived category.

By~\cite[Lemma 2.14]{plamondon}, $S$ is quasi-isomorphic, as a dg $\hat{\Gamma}(Q,W)$-module, to a complex whose underlying module
is a direct sum of shifted summands of $\hat{\Gamma}(Q,W)$.
This implies that $\Phi_{\mathbf{s}}(E)$ is perfect as a complex of $A$-modules, i.e.~a finite complex of locally free $A$-modules.  
It also implies that 
$\Phi_{\mathbf{s}}(E) \otimes K = \Phi_{\mathbf{s}}(E \otimes K)$
for each geometric point of $\spec A$.
Therefore, after restriction to each closed point, 
$\Phi_{\mathbf{s}}(E) \otimes \C$ is supported in degree $-1$, since $E\otimes \C\in T_{\mathbf{s}}$, as it admits a surjection from $\Ho^1(P_{\mathbf{n}})$, which itself is in $T_{\mathbf{s}}$ by Lemma~\ref{lem_H1_in_torsion} above.

By \cite[Lemma 4.3]{bridgeland}, this implies that  $\Phi_{\mathbf{s}}(E)$ is quasi-isomorphic to a
finite projective $A$-module supported in degree $-1$.
Therefore, after applying $\Phi_{\mathbf{s}}$ to the morphism $v'$ and passing to cohomology, we have a morphism
$$u:  \hat{J}(\mu_{\mathbf{s}}(Q),W^{\diamondsuit})_{\mathbf{n}}\otimes A \rightarrow \Ho^{-1}(\Phi_{\mathbf{s}}(E)),$$
where the target is a finite projective $A$-module.
In order for this to define an element of
$\mathcal{M}^{\sfr,\nilp, W^{\diamondsuit}}_{\nn,\mathbf{s},\gamma}\!\!(A)$, 
we need to check the condition on the slopes after base change to geometric points.  These criteria are open conditions, 
so they can be deduced from the case of closed points, where it follows already from \cite{Efi11} -- again one uses that there is a surjection $\Ho^1(P_{\mathbf{n}})\rightarrow E\otimes \Cp$, and $\Ho^1(P_{\mathbf{n}})\in T_{\mathbf{s}}[-1]=\Phi_{\mathbf{s}}^{-1}(\mathcal{C}_{<\theta_{\mathbf{s}}}\cap\hat{J}(\mu_{\mathbf{s}}(Q,W))\textrm{--mod})$.

To define the inverse to $\PHI_{\mathbf{s}}$, we argue analogously using the inverse equivalence $\Phi_{\mathbf{s}}^{-1}$.
Given a framed nilpotent module,
$u:  \hat{J}(\mu_{\mathbf{s}}(Q),W^{\diamondsuit})_{\mathbf{n}}\otimes A \rightarrow F$,
we precompose with the truncation $\hat{\Gamma}(\mu_{\mathbf{s}}(Q),W^{\diamondsuit})_{\mathbf{n}}\rightarrow \Ho^0(\hat{\Gamma}(\mu_{\mathbf{s}}(Q),W^{\diamondsuit})_{\mathbf{n}})$, apply the inverse equivalence, and pass to cohomology in degree $1$ to obtain a morphism
of $\hat{J}(Q,W)\otimes A$-modules
$$v:  \Ho^1(P_{\nn})\otimes A \rightarrow E.$$
Note that by the condition on geometric points of $u$ of Definition \ref{framednilpotentmodule}, $\Ho^0(\Phi^{-1}_{\mathbf{s}}(F))$ vanishes and so $E\cong\Phi^{-1}_{\mathbf{s}}(F)[1]$.  Flatness of $E$ follows as before, via the condition on closed points.  Furthermore, $v$ is surjective after base change to closed points; therefore it is surjective and defines an $A$-point of the
quiver Grassmannian.
\end{proof}

Finally, Proposition~\ref{unionofcomponents} is a corollary of Proposition~\ref{equivalenceofstacks}, 
Lemma~\ref{familiesofmodules}, and Lemma~\ref{indvarieties}.

\subsection{Quantum cluster positivity: results}
We can now prove our main theorem on positivity of quantum cluster transformations.
\begin{theorem}[Lefschetz condition]\label{Lefschetz}
Let $(Q,W)$ be a graded QP, with~$W$ a nondegenerate formal superpotential, and~$Q$ compatible with the skew-symmetric form $\Lambda$.  Then the quantum cluster algebra $\mathcal{A}_{\Lambda,Q}$ satisfies the Lefschetz condition of Conjecture \ref{LefProperty}.
\end{theorem}
\begin{proof}
First we remark that the presence of a Lefschetz operator with centre $N$ on a pure Hodge structure $H^*$ implies that the sequence $\dim(H^i)$, for $i$ odd or $i$ even, is symmetric unimodal, since each $H^{N-k}$ is isomorphic to $H^{N+k}$, via a chain of linear maps that factor through all $H^{N-k+2j}$ for $j<k$. It follows that $\sum \dim(H^i)x^i$ is a positive integer combination of the polynomials $P_{N,k}(q)$ defined in Conjecture~\ref{LefProperty}.

By Corollary~\ref{formaltoalg}, we may pick $W^{\diamondsuit}$ so that $\tr(W^{\diamondsuit})_{\mathbf{n},\gamma}$ is a $\Cp^*$-equivariant algebraic function on $\mathcal{M}^{\sfr}_{\mathbf{n},\mathbf{s},\gamma}$.  By Theorem~\ref{hodgetocluster}, it is enough to show that $\Ho^*(\mathcal{M}^{\sfr,\Sp,W^{\diamondsuit}}_{\mathbf{n},\mathbf{s},\gamma},\phi_{\tr(W^{\diamondsuit})_{\mathbf{n},\gamma},=1})$ carries a pure Hodge structure with a Lefschetz operator, while $\Ho^*(\mathcal{M}^{\sfr,\Sp,W^{\diamondsuit}}_{\mathbf{n},\mathbf{s},\gamma},\phi_{\tr(W^{\diamondsuit})_{\mathbf{n},\gamma},\neq 1})$ carries a pure Hodge structure, of weight one less than the invariant part, also with a Lefschetz operator. The parity part of the Lefschetz condition will then follow from Proposition~\ref{intpowers}.

By Proposition~\ref{unionofcomponents},  the reduced support of $\mathcal{M}^{\sfr,\Sp,W^{\diamondsuit}}_{\mathbf{n},\mathbf{s},\gamma}$ is projective and is a union of connected components of the reduced support of the critical locus of $\tr(W^{\diamondsuit})_{\mathbf{n},\gamma}$, an algebraic function on $\mathcal{M}^{\sfr}_{\mathbf{n},\mathbf{s},\gamma}$.
Purity of the required relative weights, and existence of Lefschetz operators, is a direct application of Corollary~\ref{cor:pure}.
\end{proof}

\begin{corollary}
\label{classpos}
Let $(Q,W)$ be a graded QP, with~$W$ a nondegenerate formal superpotential, then $\mathcal{A}_Q$ satisfies the classical positivity condition of \cite{FZ02}.
\end{corollary}
\begin{proof}
Use Lemma~\ref{quantizationprop} to quantize $\mathcal{A}_Q$, possibly adding new vertices $v_i$ to the quiver, and observe that the potential~$W$ remains nondegenerate by Remark~\ref{quantizationrem} and Proposition~\ref{adding_trees}.  Then use Theorem~\ref{Lefschetz} and specialise all new $x_i$ to 1, and also set $q^{1/2}=1$.
\end{proof}

As mentioned already, this result is now superseded by the recent paper of Lee--Schiffler~\cite{LS}.

\begin{remark}
Our purity result on spaces of stable framed representations is similar to, but logically independent from, a conjecture of Kontsevich and Efimov~\cite[Conj.6.8.]{Efi11}. This conjecture states that the result is true for generic~$W$, while we prove it for those $(Q,W)$ where~$Q$ admits an edge grading so that~$W$ becomes 
a graded potential. A generic potential on a non-acyclic quiver does not satisfy this condition.
\end{remark}
\begin{corollary} 
\label{acyclic}
If~$Q$ is mutation equivalent to an acyclic quiver, then positivity holds for $\mathcal{A}_Q$.  If $\Lambda$ is a compatible skew-symmetric form, the Lefschetz condition holds for the quantum cluster algebra $\mathcal{A}_{\Lambda,Q}$, and in particular, positivity holds for quantum cluster transformations.
\end{corollary}
\begin{proof}
This is a direct result of Theorem~\ref{Lefschetz}, using the fact that $0$ is a nondegenerate potential on an acyclic quiver by \cite[Cor.7.4]{DWZ08}.  Note that $Q$ remains acyclic even after quantization.
\end{proof}
The (quantum) positivity version of Corollary~\ref{acyclic} is obtained using quite different methods in \cite[Cor.3.3.10]{FQ}.
\begin{corollary}
If~$Q$ is mutation equivalent to a quiver $Q'$, for which there is a nondegenerate potential admitting a cut, then $\mathcal{A}_Q$ satisfies the classical positivity property.  If $\Lambda$ is a compatible skew-symmetric form, then $\mathcal{A}_{\Lambda,Q}$ satisfies the Lefschetz condition.
\end{corollary}
See the start of Section~\ref{CohoDT} for the definition of a cut.  This is just a special case of Theorem~\ref{Lefschetz} and Corollary~\ref{classpos}.  We include it since in fact there is a rich class of examples satisfying this condition, continuing the thread from the first motivations of the study of cluster algebras, in terms of Lusztig's semicanonical bases.  The interested reader should consult \cite{GLS08} for a nice introduction to this, and in particular the cluster algebra description of the dual semicanonical basis of the nilpotent radicals of Kac-Moody algebras, arising from representation theory of the preprojective algebra of the associated quiver.  For the route from this theory to our result, and in particular the description in terms of quivers with potentials of a larger class of cluster algebras, the reader should consult \cite{BIRS11}.
\begin{corollary}\label{cor_atmost4}
Let $|Q_0|\leq 4$.  Then $\mathcal{A}_Q$ satisfies the classical positivity property. If $\Lambda$ is a compatible skew-symmetric form, then $\mathcal{A}_{\Lambda,Q}$ satisfies the Lefschetz condition.
\end{corollary}
\begin{proof} We assume $|Q_0|=4$, the smaller cases being analogous or easier. Let~$W$ be a nondegenerate potential for~$Q$.  Write $W=W_{\text{simple}}+W_{\lnot \text{simple}}$, where the only cycles with nonzero coefficient in $W_{\text{simple}}$ are those that visit no vertex twice, and the cycles with nonzero coefficient in $W_{\lnot\text{simple}}$ are those that visit at least one vertex more than once.  Then one may check that this decomposition is respected by the process of taking $W_s$ for $s\in Q_0$ (this ceases to be true for $|Q_0|\geq 5$), and also that passing to a reduced QP commutes with setting $W_{\lnot\text{simple}}=0$.  Finally, the question of whether a potential $W'$ on a quiver $Q'$ is nondegenerate with respect to a single mutation at some vertex $i$ is settled entirely by $W_3$, the part of~$W$ given by 3-cycles by \cite[Prop.4.15]{FZ02}.  Putting these facts together, we deduce that $W_{\text{simple}}$ is a nondegenerate potential for~$W$.  Next, an explicit combinatorial argument shows that there is a positive weighting on the edges of~$Q$ such that every simple cycle has the same weight (this also ceases to be true for $|Q_0|\geq 5$ -- see Remark~\ref{5ormore}).  We deduce that $Q$ has a graded nondegenerate potential. 
\end{proof}

Let $S$ be a bordered surface with punctures, and marked points on the boundary.  In the paper \cite{FST06}, Fomin, Shapiro and Thurston give a way of associating a cluster algebra $\mathcal{A}_S$ to $S$.  In \cite{LF09}, Labardini-Fragoso gives an interpretation of their construction in terms of quivers with potentials, in that he defines for each $\Delta$ an `ideal triangulation' of such a surface (these form the clusters in the picture of Fomin, Shapiro and Thurston) a potential $W_{\Delta}$ on the associated quiver $Q_{\Delta}$, such that this construction commutes with mutation of clusters.  In \cite{MSW11}, the (classical) positivity conjecture is proved for these cluster algebras, by relating the coefficients occurring in cluster expansions to actual combinatorial objects.  Here we can say the following, where the first statement is a special case of the main result of \cite{MSW11}.

\begin{corollary}
Let $S$ be a surface with nonempty boundary, and marked points $M$, and let $\mathcal{A}_{Q_S}$ be the associated cluster algebra.  Then $\mathcal{A}_{Q_S}$ satisfies the classical positivity property.  Moreover, if $\mathcal{A}_{\Lambda_S, Q_S}$ is a quantization of this algebra, $\mathcal{A}_{\Lambda_S,Q_S}$ satisfies the Lefschetz property.
\end{corollary}

\begin{proof}
First assume that all the marked points $M$ lie on the boundary of $S$.  A nondegenerate potential $W_S$ is constructed for $Q_S$, the quiver associated to $S$, in \cite{LF09}.  The potential $W_S$ is cubic, and so $(Q_S,W_S)$ is a graded QP.  In the case in which some of the marked points lie in the interior of $S$ (i.e. they are punctures), it is shown in upcoming work \cite{GLFS13} of Gei{\ss}, Labardini-Fragoso and Schr\"oer that there still exists an ideal triangulation of $(S,M)$ for which all punctures are incident to only one arc, so that the nondegenerate potential of \cite{LF09} remains cubic, allowing us to deduce the more general case too.
\end{proof} 
\begin{remark}\label{rem:not_all} Corollary~\ref{Lefschetz} does not prove the quantum positivity conjecture in general, since there exist examples of quivers~$Q$ with no graded nondegenerate potential~$W$ for any grading of~$Q$. Below is a hand-made example with $9$ vertices. 

There are three $4$-cycles and two $6$-cycles in~$Q$. It can be checked by hand that a potential~$W$ on~$Q$ can only be nondegenerate, if it contains each of these cycles with nonzero coefficient. Assume that a grading of~$Q$ exists which makes~$W$ graded of degree $|W|>0$. Denote by~$n$ the sum of all the edge weights.  Calculating from the three $4$-cycles, we get $n=3|W|$, whereas calculating from the two $6$-cycles, we get $n=2|W|$, a contradiction.

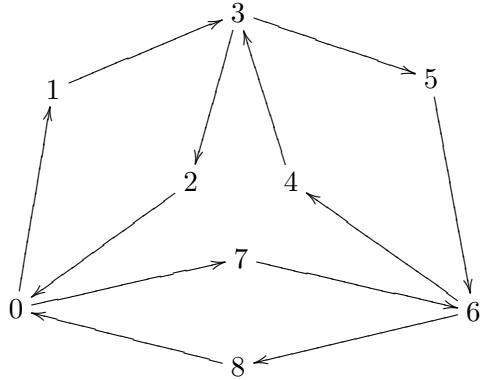
\centerline{
\begin{xy} 0;<1pt,0pt>:<0pt,-1pt>:: 
(0,112) *+{0} ="0",
(14,29) *+{1} ="1",
(66,64) *+{2} ="2",
(84,0) *+{3} ="3",
(104,64) *+{4} ="4",
(157,25) *+{5} ="5",
(173,113) *+{6} ="6",
(85,93) *+{7} ="7",
(84,134) *+{8} ="8",
"0", {\ar"1"},
"2", {\ar"0"},
"0", {\ar"7"},
"8", {\ar"0"},
"1", {\ar"3"},
"3", {\ar"2"},
"4", {\ar"3"},
"3", {\ar"5"},
"6", {\ar"4"},
"5", {\ar"6"},
"7", {\ar"6"},
"6", {\ar"8"},
\end{xy}
}

More structured examples of quivers without graded nondegenerate potentials come from punctured surfaces.  For instance, let $(S,P)$ be a punctured surface with no boundary (i.e.~let $P$ be a finite subset of the points of $S$).  A triangulation $T$ is required to have as its vertices exactly the points $P$.  The quiver associated to $T$ has cycles $l_p$ going around each puncture, as well as a cycle $l_{\Delta}$ of length three inscribed within each of the triangles $\Delta\in T$ of the original triangulation (the quiver is the dual graph to the triangulation).  The potential constructed in \cite{LF09} is just $\sum_{p\in P} l_p-\sum_{T\in\Delta} l_{\Delta}$.  In fact one can show that for large enough triangulations, none of the $l_p$ or $l_{\Delta}$ coefficients can be zero in a nondegenerate potential, i.e.~$W=\sum_{p\in P} \alpha_p l_p +\sum_{T\in\Delta} \beta_{\Delta} l_{\Delta}+(\text{other terms})$ for none of the $\alpha$ or $\beta$ equal to zero.  Since each edge of the quiver has one of the punctures on exactly one side, and is inscribed in exactly one of the triangles, one deduces as in the previous example that, assuming the existence of a graded potential $W$, the sum of the weights of the edges is equal to $|W||T|$ and also $|W||P|$; but one can have $|T|\neq|P|$.
\end{remark}
\begin{remark}
\label{5ormore}
As a first obstruction to extending the proof of Corollary~\ref{cor_atmost4}, consider the following quiver.
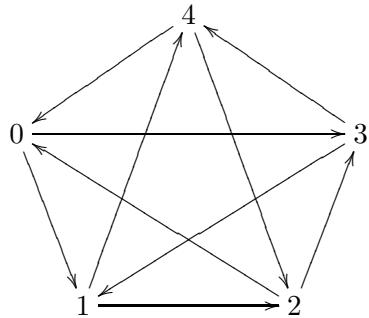
\[
\begin{xy} 0;<1pt,0pt>:<0pt,-1pt>:: 
(0,45) *+{0} ="0",
(25,110) *+{1} ="1",
(105,110) *+{2} ="2",
(130,45) *+{3} ="3",
(65,0) *+{4} ="4",
"1", {\ar"2"},
"3", {\ar"1"},
"1", {\ar"4"},
"0", {\ar"1"},
"2", {\ar"3"},
"4", {\ar"2"},
"2", {\ar"0"},
"3", {\ar"4"},
"0", {\ar"3"},
"4", {\ar"0"},
\end{xy}
\]
This quiver admits no grading such that all simple cycles have the same positive weight.  We do not know if this is a serious obstacle to strengthening Corollary~\ref{cor_atmost4} using Theorem~\ref{Lefschetz}, in other words whether there exists a quiver with 5 vertices which does not admit a graded non-degenerate potential.
\end{remark}

\begin{remark}
\label{5somos}
As a final example, consider the following quiver~\cite[Fig.9]{FM11}, arising from the Somos-5 recursion. 
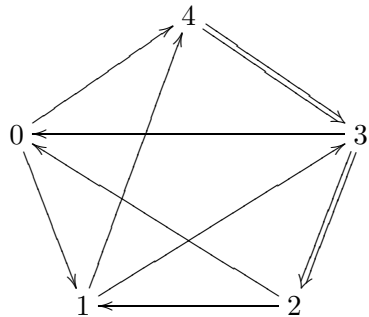
\[
\begin{xy} 0;<1pt,0pt>:<0pt,-1pt>:: 
(0,45) *+{0} ="0",
(25,110) *+{1} ="1",
(105,110) *+{2} ="2",
(130,45) *+{3} ="3",
(65,0) *+{4} ="4",
"2", {\ar"1"},
"1", {\ar"3"},
"3", {\ar@2"2"},
"2", {\ar"0"},
"0", {\ar"1"},
"3", {\ar"0"},
"4", {\ar@2"3"},
"1", {\ar"4"},
"0", {\ar"4"},
\end{xy}
\]
This quiver has no 5-cycles and, as can be checked by hand, admits a large family of gradings with all simple
cycles of the same positive weight; e.g. grading the edges $04$, $13$ and $14$  with weight one and other edges with 
weight zero gives weight one for all simple loops. It can also be checked directly that a general linear combination
of all the simple cycles gives a nondegenerate potential for this quiver. We therefore deduce 
quantum positivity for the Somos-5 quiver; classical positivity for this example was first proved 
in~\cite{Sp} (in a stronger form).
\end{remark}


\begin{thebibliography}{99} 
\bibitem{AO10} C. Amiot and S. Oppermann, Cluster equivalence and graded derived equivalence, arxiv:1003.4916.
\bibitem{Beh09} K. Behrend, Donaldson-{T}homas type invariants via microlocal geometry, Ann. of Math. (2) 170 (2009) 1307--1338.
\bibitem{BZ}  A. Berenstein and A. Zelevinsky, Quantum cluster algebras, Adv. Math. 195 (2005) 405--455.
\bibitem{bridgeland} T. Bridgeland, Equivalences of triangulated categories and Fourier-Mukai transforms, Bull. London Math. Soc. 31 (1999) 25--34.
\bibitem{Bri07} T. Bridgeland, Stability conditions on triangulated categories, Ann. of Math. 166 (2007) 317--345.
\bibitem{BIRS11} A. Buan, O. Iyama, I. Rieten and D. Smith, Mutation of cluster-tilting objects and potentials, American Journal of Mathematics 133 (2011) 835--887.
\bibitem{BBM} M. Banagl, N. Budur and L. Maxim, Intersection spaces, perverse sheaves and type IIB string theory, Adv. Theor. Math. Phys. 18 (2014) 363–-399.
\bibitem{D} B. Davison, in preparation.
\bibitem{DL99} J. Denef and F. Loeser, Motivic exponential integrals and a motivic Thom-Sebastiani theorem, Duke Math. J. 99 (1999) 285--309.
\bibitem{DWZ08} H. Derksen, J. Weyman and A. Zelevinsky, Quivers with potentials and their representations I: Mutations, Selecta Math. 14 (2008) 59--119.
\bibitem{DS} A. Dimca and B. Szendr{\H{o}}i, The Milnor fibre of the Pfaffian and the Hilbert scheme of four points on {$\mathbb C^3$},  Math. Res. Lett. 16 (2009) 1037--1055.
\bibitem{Efi11} A. Efimov, Quantum cluster variables via vanishing cycles, arXiv:1112.3601.
\bibitem{FST06} S. Fomin, M. Shapiro and D. Thurston, Cluster algebras and triangulated surfaces. Part I: Cluster complexes, Acta Mathematica 201 (2008) 83--146.
\bibitem{FZ02} S. Fomin and A. Zelevinsky, Cluster algebras I: Foundations, J. Amer. Math. Soc. 15 (2002) 497--529.
\bibitem{FM11} A. P. Fordy and R. J. Marsh, Cluster mutation-periodic quivers and associated Laurent sequences, J. Alg. Combin. 34 (2011) 19--66.
\bibitem{GLFS13} C. Geiss, D. Labardini-Fragoso and J. Schr\"oer, The representation type of Jacobian algebras, arXiv:1308.0478.
\bibitem{GLS08} C. Geiss, B. Leclerc and J. Schr\"oer, Preprojective algebras and cluster algebras, Trends in representation theory of algebras and related topics, 253-283, EMS Ser. Congr. Rep., Z\"urich, 2008.
\bibitem{Gin07} V. Ginzburg, Calabi--Yau algebras, arXiv:0612139.
\bibitem{GLM} G. Guibert, F. Loeser and M. Merle, Iterated vanishing cycles, convolution, and a motivic analogue of a conjecture of Steenbrink, Duke Math. J. 132 (2006) 409--457.
\bibitem{KKP} L. Katzarkov, M. Kontsevich and T. Pantev, Hodge theoretic aspects of mirror symmetry, in: From Hodge theory to integrability and TQFT $tt*$-geometry, 87–174, Proc. Sympos. Pure Math., 78, Amer. Math. Soc., Providence, RI, 2008.
\bibitem{Kel10} B. Keller, Cluster algebras, quiver representations and triangulated categories, in: Triangulated categories, 76--160, London Math. Soc. Lecture Note Ser. 375, Cambridge Univ. Press, Cambridge, 2010.
\bibitem{KY11} B. Keller and D. Yang, Derived equivalences from mutations of quivers with potential, Adv. Math. 226 (2011) 2118--2168.
\bibitem{FQ} Y. Kimura and F. Qin, Graded quiver varieties, quantum cluster algebras and dual canonical basis, Adv. Math. 262 (2014) 261–312..
\bibitem{king} A. King, Moduli of representations of finite dimensional algebras, Quarterly J. Math. 45 (1994) 515--530.
\bibitem{KS08} M. Kontsevich and Y. Soibelman. Stability structures, motivic Donaldson-Thomas invariants and cluster transformations, arXiv:0811.2435.
\bibitem{KS10} M. Kontsevich and Y. Soibelman. Cohomological Hall algebra, exponential Hodge structures and motivic Donaldson-Thomas invariants, Commun. Number Theory Phys. 5 (2011) 231--352.
\bibitem{Ku} V.S. Kulikov, Mixed Hodge structures and singularities, Cambridge Tracts in Mathematics 132, Cambridge Univ. Press (1998).
\bibitem{LF09} D. Labardini-Fragoso, Quivers with potentials associated to triangulated surfaces. Proc. London Math. Soc. 98 (2009) 797--839.
\bibitem{LBP90} L. Le Bruyn and C. Procesi, Semisimple representations of quivers, Trans. Amer. Math. Soc. 32 (1990) 585--598.
\bibitem{LS} K. Lee and R. Schiffler, Positivity for cluster algebras, arXiv:1306.2415.
\bibitem{MR} S. Mozgovoy and M. Reineke, On the noncommutative Donaldson-Thomas invariants arising from brane tilings, Adv. Math. 223 (2010) 1521--1544.
\bibitem{MSW11} G. Musiker, R. Schiffler and L. Williams, Positivity for cluster algebras from surfaces, Adv. Math. 227 (2011) 2241--2308.
\bibitem{Nag13} K. Nagao, Donaldson--Thomas theory and cluster algebras, Duke Math. J., 162 (2013) 1313--1367.
\bibitem{NN} K. Nagao and H. Nakajima, Counting invariant of perverse coherent sheaves and its wall-crossing, Int. Math. Res. Not. (2011) 3885--3938. 
\bibitem{Nak11} H. Nakajima, Quiver varieties and cluster algebras, Kyoto J. Math. 51 (2011) 71--126.
\bibitem{Nav} V. Navarro Aznar, Sur la th\'{e}orie de Hodge-Deligne, Inv. Math. 90 (1987) 11--76.
\bibitem{plamondon} P. Plamondon, Cluster characters for cluster categories with infinite-dimensional morphism spaces, Adv. Math. 227 (2011) 1--39.
\bibitem{S1} M. Saito, Modules de Hodge polarisables, Publ. RIMS 24 (1988) 849--995.
\bibitem{S2} M. Saito, Mixed Hodge Modules, Publ. RIMS 26 (1990) 221--333.
\bibitem{S3} M. Saito, Thom--Sebastiani theorem for mixed Hodge modules, preprint 1990/2011.
\bibitem{SS} J. Scherk and J. Steenbrink, On the Mixed Hodge Structure on the Cohomology of the Milnor Fibre, Math. Ann. 271 (1985) 641--665.
\bibitem{Sch} W. Schmid, Variation of Hodge structures: the singularities of the period mapping, Inv. Math. 22 (1973) 211--319.
\bibitem{Sp} D. E. Speyer, Perfect matchings and the octahedron recurrence, J. Alg. Combin. 25 (2007) 309--348.
\bibitem{St_limit} J. Steenbrink, Limits of Hodge structures, Inv. Math. 31 (1976) 229--257.
\bibitem{St_mhs} J. Steenbrink, Mixed Hodge structure on the vanishing cohomology, in: Real and complex singularities, 525--563, Sijthoff and Noordhoff, Alphen aan den Rijn, 1977.
\bibitem{St}  J. Steenbrink, Intersection form for quasi-homogeneous singularities, Comp. Math. 34 (1977) 211--223.
\bibitem{Sz} B. Szendr\H oi, Nekrasov's partition function and refined Donaldson--Thomas theory: the rank one case, SIGMA (2012) 088. 
\bibitem{LQT} L. Q. Thuong, Proofs of the integral identity conjecture over algebraically closed fields, arXiv:1206.5334.
\end{thebibliography}
\end{document}